\documentclass[a4paper,UKenglish,cleveref, autoref, thm-restate]{lipics-v2021}

\usepackage[utf8]{inputenc}
\usepackage{amsmath}
\usepackage{amssymb}
\usepackage{amsfonts}
\usepackage{mathtools}
\usepackage{graphicx}
\usepackage{caption}
\usepackage{empheq}
\usepackage[export]{adjustbox}
\usepackage{listings}
\usepackage{scrextend}
\usepackage{centernot}
\usepackage{amsthm}
\usepackage{tikz}
\usepackage{qtree}
\usepackage{soul}
\usepackage{algorithm}
\usepackage[noend]{algpseudocode}
\usepackage{enumitem}

\allowdisplaybreaks

\usetikzlibrary{positioning}
\usetikzlibrary{shapes}

\newcommand{\bU}{\boldsymbol{U}}

\makeatletter
\newcommand*\bigcdot{\mathpalette\bigcdot@{.5}}
\newcommand*\bigcdot@[2]{\mathbin{\vcenter{\hbox{\scalebox{#2}{$\m@th#1\bullet$}}}}}
\makeatother

\usepackage[colorinlistoftodos]{todonotes}
\DeclareMathOperator{\diam}{diam}
\DeclareMathOperator{\tw}{tw}
\DeclareMathOperator{\ltw}{ltw}

\newcommand{\N}{{{\mathbb{N}}}}

\newcommand{\m}{{\cal{M}}}
\newcommand{\f}{{\cal{F}}}
\newcommand{\W}{{\cal{W}}}
\newcommand{\dd}{{\cal{D}}}

\newcommand{\CC}{{\cal{C}}} 
\newcommand{\CT}{{\cal{T}}} 

\newcommand{\numOfV}{{N}}

\title{FPT Algorithms for Finding Near-Cliques in $c$-Closed Graphs} 

\author{Balaram Behera}{Georgia Institute of Technology, Atlanta, USA}{balaramdb@gatech.edu}{}{}

\author{Edin Husi\' c}{London School of Economics and Political Science, UK}{e.husic@lse.ac.uk}{https://orcid.org/0000-0002-6708-5112
}{}

\author{Shweta Jain}{University of Illinois, Urbana-Champaign, USA}{shwetaj@illinois.edu}{}{}

\author{Tim Roughgarden}{Columbia University, New York, USA}{tr@cs.columbia.ed}{}{Supported in part by  NSF Award CCF-1813188
 and ARO grant W911NF1910294.}

\author{C. Seshadhri}{University of California, Santa Cruz, USA}{sesh@ucsc.edu}{}{Supported by NSF DMS-2023495, CCF-1740850, 1839317, 1813165, 1908384, 1909790, and ARO Award W911NF1910294.}

\authorrunning{Behera, Husi\' c, Jain, Roughgarden, and Seshadhri} 

\Copyright{B. Behera, E. Husi\' c, S. Jain, T. Roughgarden, C. Seshadhri} 



\category{} 

\relatedversiondetails[cite=previousVersion]{A previous version of this paper is available at}{https://arxiv.org/abs/2007.09768v3} 

\acknowledgements{We would like to thank anonymous referees for their comments and suggestions.}

\nolinenumbers 

\hideLIPIcs  

\EventShortTitle{ITCS 2022}

\begin{document}

\maketitle
\begin{abstract}
Finding large cliques or cliques missing a few edges is a fundamental algorithmic task in the
study of real-world graphs, with applications in community detection, pattern recognition, and clustering. A number of effective backtracking-based heuristics for these problems have emerged from recent empirical work in social network analysis.
Given the $\mathbb{NP}$-hardness of variants of clique counting,
these results raise a challenge for \emph{beyond worst-case analysis} of these problems.
Inspired by the triadic closure of real-world graphs,
Fox et al. (SICOMP 2020) introduced the notion of $c$-closed graphs and proved
that maximal clique enumeration is fixed-parameter tractable with respect to $c$.

In practice, due to noise in data, one wishes to actually discover "near-cliques",
which can be characterized as cliques with a sparse subgraph removed.
In this work, we prove that many different kinds of maximal near-cliques
can be enumerated in polynomial time (and FPT in $c$) for $c$-closed graphs.
We study various established notions of such substructures, including $k$-plexes, complements of bounded-degeneracy and bounded-treewidth graphs.
Interestingly, our algorithms follow relatively simple backtracking procedures, analogous to what is done in practice. Our results underscore the significance of the $c$-closed graph class for theoretical understanding of social network analysis.
\end{abstract}

\newpage
\setcounter{page}{1}
\section{Introduction}
\label{section:intro}

The discovery of cliques and clique-like subgraphs is a fundamental tool in modern graph analysis, especially
for social networks. Such substructures have been used in many different applications including  community detection in social networks ~\cite{kumar1999trawling,sozio2010community}, identification of real-time stories in the news~\cite{angel2012dense} and even detection of regulatory motifs in DNA ~\cite{fratkin2006motifcut}. They have been used for graph visualization~\cite{zhang2012extracting,zhao2012large} and for creating index structures
for answering reachability and distance queries in databases~\cite{cohen2003reachability, jin20093}.

In practice, due to noise in data, one is also interested in large "near-cliques". While this is an ill-defined term, applications require cliques that are missing a small sparse subgraph. For example, incomplete cliques have been used to predict missing pairwise interactions ~\cite{yu2006predicting} and for identifying functional groups~\cite{han2007identifying} in a protein interaction network. They have been used for community detection~\cite{zhu2020community} and for detecting test collusion ~\cite{belov2021graph}. Recent works have used the fraction of near-cliques to $k$-cliques to define higher
order variants of clustering coefficients~\cite{yin2017local}. A common notion is that of $k$-plexes (a clique minus a subgraph with degree bound $k$). They have been used in community detection ~\cite{wang2017query, bacciu2021k}, for partitioning of sparse biological networks~\cite{grbic2019variable}, and  for determining molecular similarity~\cite{hernandez2016novel}. 

From a worst-case standpoint, even the simpler problem of maximum clique is a notoriously difficult computational problem. Even getting $O(n^{1-\delta})$-approximations is $\mathbb{NP}$-hard~\cite{haastad1999clique, zuckerman2006linear}, and it is hard to non-trivially approximate even with algorithm parameterized by solution size~\cite{chalermsook2017gap}.
On the other hand, there have been many recent successes in clique enumeration/approximation in the data mining community~\cite{jain2020power,jain2020provably,li2020ordering,danisch2018listing,chen2018mining,dhulipala2021parallel,jain2017fast}. Many of these results employ backtracking heuristics~\cite{jain2020power,jain2017fast,danisch2018listing,dhulipala2021parallel}. These algorithms can even get the exact maximum clique for graphs
with millions of edges. Moreover, the basic backtracking techniques work for approximating counts of cliques missing a few edges~\cite{jain2020provably,wu2007parallel,bentert2019listing,song2015method}. 

This gap between theory and practice is the main focus of our work. \emph{Can we prove the existence of efficient (hopefully, backtracking) algorithms for near-clique discovery, assuming the input has "reasonable" properties of social networks?}

The starting point for our work is the recent notion of \emph{$c$-closed graphs}, defined by Fox et al.~\cite{DBLP:conf/icalp/FoxRSWW18,fox2020finding}. 
Triadic closure -- the property that friends of friends are often friends -- is a well-observed property of social networks. A $c$-closed graph has the property that two vertices sharing at least $c$ common neighbors are connected by an edge. Fox et al. empirically show that real-world social network are often (or approximately) $c$-closed for small values of $c$. Theoretically, they proved that maximal clique
enumeration can be done in time $2^{O(c)}n^2$, and is hence fixed parameter tractable (FPT) in $c$. (The basic brute force algorithm can be shown to run in $O(n^{c})$ time.)

\subsection{Main results} \label{sec:mainresult}

Our focus is on counting the number of maximal near-clique structures, which we can roughly
define as "a clique minus a sparse subgraph", or alternately, the complement of
a sparse subgraph. The input graph $G$ has $n$ vertices, $m$ edges, and is assumed to be $c$-closed.

We define the various pattern subgraphs that will be counted.
We begin with the classic notion of a $(d+1)$-plex.

\begin{definition}[$(d+1)$-plex,~\cite{seidman1978graph}]
  A subset of vertices $S$ is called a \emph{$(d+1)$-plex} 
  if each $v\in S$ is adjacent to all but at most $d$ vertices of $S$ (excluding itself).
\end{definition}

Observe that a $(d+1)$-plex is precisely the complement of a graph with maximum degree at most $d$. 
Our first result is that enumerating maximal $(d+1)$-plexes (for constant $d$) in an input $c$-closed
graph is FPT in $c$. 

\begin{theorem}\label{thm:fasterBoundedDegree}
	For $c$-closed graphs and a fixed $d \ge 0$, 
	there is an algorithm running in time $O(n^{2d} \cdot \kappa_d^{c} \cdot p(c))$ for enumerating $(d+1)$-plexes, 
	where $\kappa_d < 2$ is the root of $x^{d+4} - 2x^{d+3} + 1=0$; and for a polynomial $p$.
 	For $2$-plexes, a stronger bound $O(n^{2} \cdot 10^{c/5} \cdot p(c))$ applies.
\end{theorem} 

We go further and show analogous results for other patterns that can be expressed as complements
of sparse graphs. A pattern has \emph{bounded co-degeneracy} if the degeneracy of the complement
is bounded. The degeneracy can be thought of as a more robust notion of maximum degree, and
has
a significant role in social network analysis. Bounded co-degenerate graphs are 
a natural generalization of $(d+1)$-plexes. Analogously, we also consider counting maximum
\emph{bounded co-treewidth} graphs. 

\begin{theorem} \label{thm:max-degen} For $c$-closed graphs and a fixed $d \ge 0$, 
	there is an algorithm running in time $O(n^{2d+4}4^c)$ that outputs 
all maximal induced subgraphs with co-degeneracy $d$ in an input $c$-closed graph.
\end{theorem}

\begin{theorem}\label{thm:max-treewidth}
	For $c$-closed graphs and a fixed $t \ge 0$, there is an algorithm running in time $O(n^{t+4}2^{2c})$ that outputs all maximal induced subgraphs with co-treewidth $\le t$.
\end{theorem}

The exponential dependence $n^d$ in Theorems~\ref{thm:fasterBoundedDegree} and~\ref{thm:max-degen} is necessary, as is the dependence $n^t$ in Theorem~\ref{thm:max-treewidth} as we show with examples.

We note that not all natural notions of ``co-sparse'' subgraphs lead
to FPT bounds.  For example,
the maximal subgraphs with bounded \emph{average} co-degree cannot be
listed by an FPT algorithm, even for average co-degree of at most~2.

\begin{example}\label{example1}
  Let $ \ell \in \N$ and let $c = \frac{\ell}{2}(\ell+1) + 1$.  By the
  hand-shaking lemma, a subgraph $G[S]$ has average co-degree at most
  $2$ if and only if $G[S]$ contains at most $|S|$ non-edges.
  Consider a graph $G$ consisting of a clique $K$ on $c-1$ vertices
  and an independent set $I$ on $n$ vertices, where any vertex in $I$
  is adjacent to every vertex in $K$. $G$ is $c$-closed since
  any two non-adjacent vertices are adjacent only to $K$.  Note that
  $G$ contains exactly $n + c-1$ vertices.

  Let us show that the number of maximal subgraphs $G[S]$ with at most
  $|S|$ non-edges is at least~$n^{\ell}$ and hence not FPT with
  respect to~$c$.  In particular, consider a set of the form $S\cup K$
  where $S\subseteq I$.  If $|S| = s$, then the number of non-edges in
  $G[S\cup K]$ is exactly $s(s-1)/2$.  By the choice of~$c$, any set
  $S$ of size $\ell$ is a maximal subgraph with at most $|S|$
  non-edges.  Thus, there are at least
  $O(n^{\ell}) \approx O(n^{\sqrt{2c}})$ maximal subgraphs $G[S]$ with
  at most $|S|$ non-edges.
\end{example}

{\em The backtracking connection:} One of the first steps in proving the above theorems
is a different, simpler proof that maximal clique enumeration is FPT in $c$. (This
is the main result of Fox et al.~\cite{fox2020finding}.) Typical backtracking algorithms exhaustively and incrementally build candidates for solutions until they have discovered all candidates. We analyze a simple
backtracking procedure that finds cliques (Section~\ref{section:cliques}), and give a bound on its running time. 
Moreover, we use this result to show that maximal bounded co-degenerate subgraphs
can be enumerated efficiently.
We consider these proofs as mathematical justification for the empirical success of backtracking algorithms, 
and see our results as ``beyond the worst-case analysis'' results~\cite{roughgarden_2021}.

\smallskip
{\em Organization:} In Section~\ref{section:results}, we describe our results in more detail. Section~\ref{section:related} covers related work. Section~\ref{section:preliminaries} describes the definitions and terms required for the proofs. Sections~\ref{section:cliques},~\ref{section:plexes}, ~\ref{section:degeneracy} and~\ref{section:treewidth} respectively gives proofs for FPT bounds for cliques, $(d+1)$-plexes, bounded co-degeneracy and bounded co-treewidth graphs.

\subsection{Discussion of results}
\label{section:results}

\smallskip \noindent \textbf{Cliques} We first provide a simple proof that uses a backtracking tree to show that the number of maximal cliques is bounded by $O(cm2^c)$ where $m$ represents the number of edges in the complement graph. (Fox et al. prove a bound of $\min{\{3^{(c-1)/3}n^2,4^{(c+4)(c-1)/2}n^{2-2^{1-c}}\}}$). We convert this result into a simple backtracking algorithm for enumerating maximal cliques that runs in time $O(cmn^22^c)$. Although the running time bound we obtain is slightly worse than that of Fox et al., the algorithm and proof are simpler, in particular, as  Fox et al. black-box clique enumeration. We also believe that our proof provides theoretical understanding for the practical efficiency of common backtracking methods, such as the Bron-Kerbosch algorithm \cite{bron1973algorithm}
and a recent work of Jain-Seshadhri~\cite{jain2020power}.

\noindent\textbf{Two approaches} For the other dense subgraph types, we do the following: for each type, we provide structural results bounding the maximum possible number of maximal subgraphs of that type. 
Our results come in two flavors. In one flavor, the \emph{backtracking approach}, we show that any subgraph of that type can be split into parts which are either bounded in size or are cliques. 
For parts that are cliques, we use the simple backtracking algorithm for counting cliques mentioned above. For parts that are not cliques (and are thus bounded in size), we simply find candidate vertices for each part, enumerate all subsets of these candidate sets and combine them to give a set of subgraphs that is a superset of the set of all maximal subgraphs of that type (for cliques, $k$-plexes and co-degenerate subgraphs, there exist simple tests for checking if a subgraph is a maximal subgraph of that type).  Because the parts are of bounded size, we get FPT bounds for the size of this superset. 
In some cases (cliques and $d+1$-plexes), this approach leads to slightly worse exponential factors than bounds obtained using the second approach, but leads to simple algorithms that are easy to describe. Indeed, the enumeration algorithms follow from the structural results; obtaining the structural results is the main challenge. Interestingly, the algorithms obtained using this approach have significant portions that use backtracking, reflecting the fact that backtracking has proven to be effective in practice. 

In the other flavor, the \emph{three-step  approach}, we use the approach taken by Fox et al. for proving their result for maximal cliques. We view their proof as being composed
of three parts.
The first part uses a combinatorial bound on the number of maximal
cliques, the classic Moon-Moser theorem~\cite{moonmoser,millerMuller}.
This theorem states that the number of maximal cliques in an arbitrary
$\numOfV$-vertex graph is bounded above by $3^{N/3}$ (with a matching
lower bound furnished by a complete $(N/3)$-partite graph).  

The second and most interesting part of the proof exploits the
$c$-closed condition to translate the Moon-Moser theorem into an FPT
bound of at most $n^2 3^{(c-1)/3}$ maximal cliques in a $c$-closed
graph with $n$ vertices.  Roughly, this step of the proof works as
follows.  For (almost) every maximal clique, one can identify two
non-adjacent vertices such that the clique is contained in the common
neighborhood of the two vertices.  Such a maximal clique in the
original graph is also maximal in an induced subgraph on at most $c-1$
vertices, by the $c$-closure property.  The upper bound follows by
applying the Moon-Moser theorem to these subgraphs (of which there is a
polynomial number), each of size at most $c-1$.

The third step is to translate the FPT combinatorial bound on the
number of maximal cliques into an FPT algorithm for enumerating them.
For the case of cliques, there is a well known
algorithm~\cite{tsukiyama1977new} that can be used to list all maximal
cliques in $O(mn)$ time per
clique.\footnote{Replacing the Moon-Moser
  bound with the trivial bound of~$2^N$ would also lead to an FPT
  result, albeit one that is exponentially worse.
  Fox et al.~\cite{DBLP:conf/icalp/FoxRSWW18,fox2020finding} also prove an
  incomparable bound with better dependence on~$n$ ($n^{2-2^{1-c}}$)
  but worse dependence on~$c$ ($4^{(c+4)(c-1)/2}$).}

Thus, for proofs using the three-step approach, we use the same three-part
framework outlined above for the special case of cliques:
\begin{enumerate}

	\item \textbf{Combinatorial bound:} Find an upper bound on the
          number of maximal dense subgraphs in an arbitrary
          $\numOfV$-vertex graphs, in the spirit of the Moon-Moser theorem.
(Either relying on an existing
          bound or proving a new one from scratch.)

	\item \textbf{FPT bound:} Exploit the $c$-closed condition to
          translate the combinatorial bound into an FPT-type upper bound (with parameter $c$) on the number of maximal dense subgraphs 
	in a $c$-closed graph on $n$ vertices.

	\item \textbf{Enumeration:} Give an FPT enumeration algorithm for
          listing all maximal dense subgraphs in a $c$-closed graph.
          (Either relying on an existing enumeration algorithm or
          devising a new one.)

\end{enumerate}

We describe our contributions in more detail below:

\smallskip \noindent \textbf{$(d+1)$-plexes}\footnote{Similar result for $(d+1)$-plexes was proved independently and concurrently with the previous version of this paper by Koana,
Komusiewicz, and Sommer~\cite{KKS}. 
The results in~\cite{KKS,koana_et_al:LIPIcs:2020:13364} apply more generally to the class of
weakly $c$-closed graphs defined in
\cite{DBLP:conf/icalp/FoxRSWW18,fox2020finding}
(The paper~\cite{KKS} also includes several results showing polynomial-size
kernels for various problems in weakly $c$-closed graphs, an important
direction that is not pursued here.)} A subset $S\subseteq V(G)$ is called a {\em $(d+1)$-plex} if every
vertex $v\in S$ is non-adjacent to at most $d$ other vertices in $S$.
Equivalently, a subset $S$ is a $(d+1)$-plex if $G[S]$ has co-degree
at most $d$. Thus, a clique is $1$-plex.  This is a common relaxation of
cliques used in practice~\cite{fortunato2010community, seidman1978graph}.  For each
fixed~$d$, we give an FPT algorithm for enumerating {$(d+1)$-plexes}. In general graphs, an FPT algorithm for finding a largest $(d+1)$-plex is impossible (assuming P $\neq$ NP)
\cite{lewis1980node}.

For the backtracking approach, we show that every maximal $(d+1)$-plex is either a maximal clique, or contains a pair of non-adjacent vertices $(u,v)$ such that the $(d+1)$-plex can be split into two parts -- one part of size at most $2d-2$ consisting of vertices that are non-adjacent to either $u$ or $v$, and the other of size at most $c$ consisting of (a subset of) common neighbors of $u$ and $v$. Since the number of pairs of non-adjacent vertices in the given $c$-closed graph is equal to the number of edges in its complement graph, $m$, this gives the maximum number of maximal $(d+1)$-plexes as $O(mn^{2d-2}2^c)$ and the enumeration algorithm follows. 

For the three-step approach, we use $\m_d(\numOfV)$ -- the maximum number of maximal $(d+1)$-plexes in an
$\numOfV$ vertex graph.  (Equivalently, $\m_d(\numOfV)$ is the
number of maximal subgraphs of degree at most $d$ in an $\numOfV$
vertex graph.) For the combinatorial bound we need an upper bound on $\m_d(\numOfV)$. A recent result shows that for every fixed $d$ there
is a constant $\kappa_d < 2$ such that
$\m_d(\numOfV) \le \kappa_d^{\numOfV}$~\cite{zhou2020enumerating}.

 Determining a tight bound for $\m_d(\numOfV)$ appears to be 
challenging.  To the best of our knowledge, the only tight bound is
the Moon-Moser theorem stating that
$\m_0(\numOfV) \le 3^{\numOfV/3} \approx 1.442^{\numOfV}$.  One of our
contributions is to give a tight bound for $\m_1(\numOfV)$:
$\m_1(\numOfV) \le 10^{\numOfV/5} \approx 1.585^{\numOfV}$.  This
result is presented in Appendix~\ref{section:combinatorics}, and
requires a much more involved proof than the Moon-Moser theorem (see Appendix~\ref{sec:proofOfMoonMoser} for a short
proof of the Moon-Moser~theorem).\footnote{
  The induced subgraphs with maximum degree at most one are also called \emph{dissociation sets}~\cite{yannakakis1981node}. 
  Thus, we show that the number of maximal dissociation sets in an $\numOfV$-vertex graph is at most $10^{\numOfV/5}$.
}

In the second step of the three-step approach, we give an FPT bound with a smaller (than in the case of backtracking) exponential factor $O(n^{2d} \cdot \kappa_d^{c})$ using a more careful analysis of the structure of a $(d+1)$-plex.
(Example~\ref{examplePlexes} shows that the exponential dependence~$n^d$ is necessary.)  Moreover, using the tight bound for
$\m_1(\numOfV)$ we give a stronger bound $O(n^{2} \cdot 10^{c/5})$ for
the number of maximal $2$-plexes in a $c$-closed graph on $n$
vertices.  

To convert the tighter bound into an enumeration algorithm and complete the third step, the simplest approach is
to apply black-box one of the recent polynomial delay
algorithms for efficiently listing $(d+1)$-plexes~\cite{berlowitz2015efficient,cao2020enumerating}. 
E.g., Berlowitz et al.~\cite{berlowitz2015efficient} give an algorithm which enumerates all maximal $(d+1)$-plexes 
in time $O((d+1)^{2d+2} p(n))$ per maximal $(d+1)$-plex, where $p(n)$ is a polynomial in~$n$.
By the FPT bound, the enumeration algorithm runs in FPT time. 
However, we can obtain a better running time by translating our proof
of the FPT bound into a bespoke enumeration algorithm.

\smallskip \noindent \textbf{Bounded co-degeneracy}
We say that a graph has {\em co-degeneracy} at most $d$ if its
complement is $d$-degenerate.  (Recall that a graph is $d$-degenerate
if every induced subgraph has at least one vertex with degree at
most~$d$.)  In Section~\ref{section:degeneracy} we give, for each
fixed $d$, FPT algorithms for enumerating maximal subgraphs with
co-degeneracy at most $d$.

For the backtracking approach, we first show that every subgraph with bounded co-degeneracy is either a clique, or the degeneracy ordering of the complement of the subgraph contains an edge that splits the subgraph into three parts; two of whose sizes are bounded ($2d-2$ and $c$, respectively) and the third is a maximal independent set (in the complement graph) which can be discovered using the  algorithm for enumerating cliques. This gives a bound of $O(cm^2n^{2d}4^c)$ on the number of maximal subgraphs with co-degeneracy $d$ and an enumeration algorithm follows. 

For the three-step approach, for the combinatorial bound, we define $\dd_d(\numOfV)$ to be the
maximum number of maximal subgraphs with co-degeneracy at most $d$ in
an arbitrary $\numOfV$-vertex graph.  
For every fixed
$d$ there is a constant $\gamma_d<2$ such that
$\dd_d(\numOfV) \le \gamma_d^{\numOfV}$,
see~\cite{pilipczuk2012finding}.

For the FPT bound, we show that the number of maximal subgraphs with
co-degeneracy at most $d$ is at most
$O(n^{8d}\cdot \dd_d(2dc)) \le O(n^{8d}\cdot \gamma_d^{2dc})$.  The
idea is to show that there are two types of maximal subgraphs with
co-degeneracy at most $d$: either they have the structure of a
generalized co-star, or we can find $2d$ pairs of non-adjacent edges such
that the maximal subgraph is contained in the common neighborhoods of
these non-adjacent pairs and an additional $4d$ vertices.  Counting
generalized stars reduces to counting cliques, and we control the
other case using the $c$-closed condition.

An FPT algorithm is 
 obtained by applying the recent enumeration
algorithm~\cite{conte2019proximity} that lists all maximal subgraphs
with bounded degeneracy in time $O(mn^{d+2})$ per maximal subgraph.

\smallskip \noindent \textbf{Bounded co-treewidth}
A graph is said to have {\em co-treewidth} at most $t$ if its
complement has treewidth at most $t$.  The class of graphs with
co-treewidth at most $t$ is denoted by $\CT_t$.  In
Section~\ref{section:treewidth}, we give, for each fixed $t$, FPT
algorithms for enumerating {$\CT_t$-graphs} using (only) the three-step approach.

Obtaining non-trivial combinatorial bounds on the number of maximal
subgraphs with (co-)treewidth at most $t$ in an arbitrary $N$-vertex
graph is an open question in graph theory, so we use
the trivial upper bound of
$2^{\numOfV}$.  
(In any case, there are no known polynomial-delay algorithms for listing
subgraphs of bounded (co-)treewidth that would allow us to
algorithmically exploit (black-box) the savings that a better bound
would give us.)

For our FPT bound, we show that
for almost every maximal subgraph of bounded co-treewidth we can
either find two pairs of non-adjacent vertices and show that the
subgraph is contained in the common neighborhoods of these two pairs
(plus $t$ additional vertices), or else that the subgraph is a
generalized co-star.  In the former case we use the $c$-closure
condition and reduce the latter case to counting maximal cliques in
smaller graphs.  We show that there are $O(n^{t+4}2^{2c})$
maximal subgraphs with co-treewidth at most $t$.  Exponential
dependence $n^t$ is necessary, even when~$c=1$ (Example~\ref{exampleTW}).

While there are no known polynomial-delay enumeration algorithms for
listing maximal subgraphs of bounded \mbox{(co-)treewidth}, we show
how to turn our FPT bound into an FPT algorithm for enumerating {
  $\CT_t$-graphs}.

We also extend these results to the 
subgraphs of bounded
\emph{local} co-treewidth (Appendix~\ref{section:localTreeWidth}).

\subsection{Further related work}
\label{section:related}

\noindent \textbf{Polynomial-time solvable special cases of the 
{\sc Maximum Clique} problem and its generalizations in hereditary
graph classes}
The problems we consider generalize the fundamental {\sc Maximum
  Independent Set} and {\sc Maximum Clique} problems.  It is well
known that polynomial-time and fixed-parameter tractability results
for these problems require significant restrictions on the allowable
input graphs.
For example,
it is known that {\sc Maximum Independent Set} is NP-hard already for
subcubic graphs, and for $H$-free graphs (for $H$ connected) whenever
$H$ is not a path nor a subdivision of the claw ($K_{1,3}$)~\cite{alekseev1982effect}. 
Similarly, the problem is $W[1]$-hard when parameterized by the
solution size for $H$-free graphs whenever $H$ is not a suitable
generalization of a path or a subdivision of the
claw~\cite{DBLP:journals/algorithmica/BonnetBCTW20} (obtained by
replacing each vertex by a clique); in fact, the problem does not even admit
an FPT constant-factor approximation 
for these graph classes (assuming Gap ETH)~\cite{dvovrak2020parameterized}.
Known polynomial-time solvable special cases of the
{\sc Maximum Independent Set} problem include input graphs that are
perfect 
(including (co-)chordal and (co-)bipartite graphs), $P_6$-free
graphs~\cite{lokshtanov2017independence}, fork-free
graphs~\cite{lozin2008polynomial}, and other highly restricted classes~\cite{graphclasses,chudnovsky2020maximum, chudnovsky2019maximum,husic2019polynomial}.

\smallskip \noindent \textbf{Real worlds graphs}
It is widely accepted that the real-world graphs possess several nice properties that differentiate them from arbitrary graphs. 
The established ones include heavy-tailed degree distributions, 
a high density of triangles and communities,
the small world property (low diameter), and triadic closure. 
Over the years there has been a lot of significant and influential work trying to capture 
the special structure of real-world graphs. 
The literature is almost entirely focused on the generative (i.e., probabilistic) models. 
A few most popular ones include
preferential attachment~\cite{barabasi1999emergence}, the copying model~\cite{kumar2000stochastic}, 
Kronecker graphs~\cite{leskovec2010kronecker}, the Chung-Lu random graph model~\cite{chung2002average, chung2002connected}, 
with many new models introduced every year.
For example, already in 2006, the survey by Chakrabarti and Faloutsos~\cite{chakrabarti2006graph} examines 23 different models.
Generative approaches are very enticing as they, by definition, give an easy way of producing synthetic data,
 and are a good proxy for studying random processes on graphs. 
On the other hand, if one is to design an algorithm for real-world graphs with good worst-case guarantees, 
a hard choice of the exact model arises as there is a little consensus about which of the many models is the “right” one, if any. 

An idea is to find algorithms that are not suited to any specific generative model, 
but only assume a deterministic condition.
In other words, isolate a parameter of the real-world graphs that differentiates them from arbitrary graphs 
and use it give stronger guarantees for particular algorithms/problems.
Fox, Roughgarden, Seshadhri, Wei, and Wein~\cite{DBLP:conf/icalp/FoxRSWW18,fox2020finding} took this approach and introduced the class of $c$-closed graphs, 
where they showed that the maximum clique  problem is FPT when parameterized by $c$. 

There are only a few other algorithmic results in the same spirit.
Notably, several problems can be solved faster for graphs with a power-law degree distribution:
Barch, Cygan, \L acki, and Sankowski~\cite{brach2016algorithmic} gave faster algorithms for transitive closure, maximum matching, determinant,
PageRank and matrix inverse; and 
Borassi, Crescenzi, and Trevisan~\cite{borassi2017axiomatic} gave faster algorithms for
diameter, radius, distance oracles, and computing the most ``central'' vertices by assuming additional axioms satisfied
by real-world graphs.

Motivated by triadic closure, Gupta, Roughgarden, and Seshadhri~\cite{gupta2016decompositions} define triangle-dense graphs 
and proved relevant structural results. 
Informally, they proved that if a constant fraction of two-hop paths are closed
into triangles, then (most of) the graph can be decomposed into clusters with diameter at most $2$.

\noindent \textbf{$c$-closed graphs} The $c$-closed graph model was introduced by Fox et al.~\cite{fox2020finding} (see book chapter in  ~\cite{roughgarden_2021} by some of the authors). After Fox et al. introduced $c$-closed graphs, Koana, Komusiewicz, and Sommer wrote several papers further exploting $c$-closure to design FPT algorithm for hard problems. In~\cite{koana2020exploiting} they showed that 
the dominating set problem, 
the induced matching problem, and the irredundant set problem admit kernels 
of size $k^{O(c)}$, $O(c^7 k^8)$, $O(c^{5/2} k^3)$ respectively; where $k$ is the size of the solution. In~\cite{DBLP:conf/isaac/KoanaKS20}, they show that enumerating maximal bicliques and $(d+1)$-plexes, is FPT with respect to $c$ and study fixed parameter tractability of related hard problems with respect to the parameter $c$ and size of the solution.
In~\cite{DBLP:journals/corr/abs-2103-03914}, they give the kernels for Capacitated Vertex Cover, Connected Vertex Cover, and Induced Matching of sizes $k^{O(c)}$, and $(ck)^{O(c)}$, respectively. Moreover, Koana and Nichterlein~\cite{DBLP:journals/dam/KoanaN21} explore the fixed parameter tractability of enumerating small induced subgraphs in a $c$-closed graph. 

We note that the densest subgraph problem is trivially solvable in polynomial time for $c$-closed graph when $c=1$, and NP-hard already for $c=2$, see~\cite{DBLP:journals/algorithmica/RamanS08}.

\smallskip \noindent \textbf{$(d+1)$-plexes}
The maximal cliques often fail to detect cohesive subgraphs. 
To address the issue, Seidman and Foster~\cite{seidman1978graph} in 1978 introduced the notion of $(d+1)$-plex.
We refer the reader to~\cite{wu2007parallel, mcclosky2012combinatorial,pattillo2013clique,berlowitz2015efficient, conte2017fast,bentert2019listing} and references therein for an overview of the literature. 
The literature is mostly focused on heuristic algorithms for finding large $(d+1)$-plexes or enumerating (several) maximal $(d+1)$-plexes 
without providing any worst-case guarantees. 
For example, recently Conte, Firmani, Patrignani, and Torlone~\cite{conte2019shared} gave a novel approach for the detection of 2-plexes.
We point out that Lewis and Yannakakis~\cite{lewis1980node} proved that the problem of finding a maximum $(d+1)$-plex is NP-hard for any fixed $d$. Alternate proof is given in~\cite{balasundaram2011clique}.

\smallskip \noindent \textbf{Counting and enumerating maximal subgraphs}
Counting (maximal) induced subgraphs in an arbitrary $N$-vertex graph is a crucial part when it comes to design of faster exact algorithms. 
We mention a few related results.
Moon and Moser~\cite{moonmoser} and also Miller and Muller~\cite{millerMuller} 
prove that the number of maximal cliques (equivalently maximal independent sets) 
in a graph on $\numOfV$ vertices is at most $3^{\numOfV/3}$. 
Tomita, Tanaka and Takahashi~\cite{tomita2006worst} gave an algorithm for finding a maximum clique by enumerating all maximal cliques in time $O(3^{\numOfV/3})$. 

Gupta, Raman and Saurabh~\cite[Theorem 4]{rRegular} show that the number of maximal $1$-regular induced graphs in an $\numOfV$-vertex graph is at most $10^{\numOfV/5}$ and gave an algorithm for finding a maximum such subgraph with similar running time.
Note that in any graph, the number of maximal induced matchings is not larger than the number of maximal induced subgraphs with degree at most $1$.
Therefore, it is somewhat surprising that the number of maximal induced subgraphs with degree at most $1$ 
is also bounded by $10^{\numOfV/5}$, as we show in Appendix~\ref{section:combinatorics}.
The same paper~\cite{rRegular} shows that for each integer $r$ there is a constant $\rho_r<2$, 
such that the number of maximal $r$-regular graphs in an $\numOfV$ vertex graph is at most $\rho_r^{\numOfV}$.

Zhou, Xu, Guo, Xiao, and Jin~\cite{zhou2020enumerating} show that for each $d$ there is a constant $\kappa_d<2$ 
such that all maximal $(d+1)$-plexes can be enumerated in time $O(\kappa_d^{\numOfV} \numOfV^2)$.
Implicitly, they also show that the number of maximal $(d+1)$-plexes is at most $\kappa_d^{\numOfV}$, i.e., $\m_d(\numOfV) \le \kappa_d^{\numOfV}$.

Pilipczuk and Pilipczuk~\cite{pilipczuk2012finding} show that for every fixed $d$ there is a constant $\gamma_d<2$ such that the number of maximal induced $d$-degenerate subgraphs in a graph on $\numOfV$ vertices is at most $\gamma_d^{\numOfV}$, i.e., $\dd_d(\numOfV) \le \gamma_d^{\numOfV}$.

\section{Preliminaries and complementary terminology}
\label{section:preliminaries}
We consider finite, simple, undirected graphs. Let $G = (V, E)$ be a graph.
We write $uv  \in E(G)$ for an edge $\{u, v\} \in E(G)$ and we say that 
the vertices $u$ and $v$ are \emph{adjacent} or that $u$ is a \emph{neighbor} of $v$ and vice versa. 
If $w\in N_G(u) \cap N_G(v)$ we say that $w$ is a \emph{common} neighbor of $u$ and $v$. 
For a vertex $v\in V(G)$ we denote by $N_G(v) = \{u \in V(G) : uv \in E(G)\}$ the \emph{neighborhood} of $v$ in $G$ and 
$N_G[v] = N_G(v) \cup \{v\}$ the \emph{closed neighborhood} of $v$ in $G$.
For $U\subseteq V(G)$, we define $N_G(U) = \cup_{u\in U} N_G(u) \setminus U$ and $N_G[U] = N_G(U) \cup U$. 
For simplicity, if the set $U$ is given implicitly as a collection of vertices $u_1, \dots, u_{\ell}$ we write $N_G(u_1, \dots, u_{\ell})$
instead of $N_G(\{u_1, \dots, u_{\ell}\})$, and similarly for $N_G[u_1, \dots , u_{\ell}]$.
We drop the subscript $G$ when the graph is clear from the context.

Let $W\subseteq V(G)$.
The induced subgraph $G[W]$ is defined as the graph $H = (W, E(G) \cap \binom{W}{2})$, where $\binom{W}{2}$
is the set of all unordered pairs with elements in $W$.
The graph $G[V(G)\setminus W]$ is also denoted as $G\setminus W$.
Set $W$ is \emph{separator} in $G$ if $G\setminus W$ has strictly more connected components than graph $G$.
A connected component is \emph{non-trivial} if it contains at least two vertices (equivalently at least one edge).
The diameter of $G$, denoted $\diam(G)$, is the length of a longest shortest path among two vertices in $G$.
If $G$ is disconnected, then $\diam(G) = \infty$.

The \emph{complement} of a graph $G= (V,E)$ is the graph $\overline G:= (V, \binom{V}{2} \setminus E)$.
We say that $W$ is a clique (in $G$) if for any two vertices $u,v \in W$ we have $uv \in E(G)$.
A set $I$ is an \emph{independent set} in $G$ if $I$ is a clique in $\overline G$.
A set $U \subseteq V(G)$ is a vertex cover in $G$ if $V(G)\setminus U$ is an independent set in $G$.
\begin{sloppypar}
The \emph{degree} of $v$ in $G$ is $\deg_G(v) = |N_G(v)|$,
and the \emph{maximum degree} of $G$ is \mbox{$\Delta(G) = \max_{v\in V(G)} \deg_G(v)$}.
Graph is $d$-\emph{degenerate} (has degeneracy at most $d$) if every induced subgraph of $G[S]$ contains a vertex $v$ such that $\deg_{G[S]}(v) \le d$. 
\end{sloppypar}
\begin{definition}[Treewidth,~\cite{robertson1986graph}]
Let $G$ be a graph. A \emph{tree decomposition} of $G$ is a pair $(T, \W)$, where $T$ is a tree and $\W = \{W_t \subseteq V(G) : t \in V(T)\}$
is a set of \emph{bags} satisfying
\begin{itemize}
	\item $\cup_{t\in V(T)} W_t = V(G)$ and for every edge $uv$ in $G$ there is bag $W_t$ containing $u$ and $v$; and
	\item if $t, t', t'' \in V(T)$ and $t'$ lies on the path between $t$ and $t''$ in $T$, then $W_t \cap W_{t''} \subseteq W_{t'}$.
\end{itemize}
The \emph{width} of $(T, \W)$ is $\max_{t\in V(T)} (|W_t| -1)$. 
The \emph{treewidth} of $G$, denoted $\tw(G)$, is the smallest number $t$ such that there is a tree decomposition $(T, \W)$ of $G$ with width $t$.
\end{definition}

Co-degree, co-treewidth, and co-degeneracy refer to the degree, treewidth and degeneracy in the complement graph, respectively.
\begin{definition}[$c$-closed,~\cite{DBLP:conf/icalp/FoxRSWW18}]
A graph $G$ is $c$-\emph{closed} if any two non-adjacent vertices have at most $c-1$ common neighbors.	
\end{definition}
Finding the smallest $c$ for which a given graph $G$ is $c$-closed
can be done by squaring the adjacency matrix in $O(n^{\omega})$ time, 
where $\omega < 2.373$ is the matrix multiplication exponent. 

A problem is said to be \emph{fixed-parameter tractable} with respect to a parameter $k$ if there is an algorithm that
solves it in time $O(f(k) n^{\alpha})$ where $f$ can be an arbitrary function and $\alpha$ is a constant, 
for more details on parameterized algorithms and complexity we refer to~\cite{cygan2015parameterized}.
Throughout the paper, unless otherwise stated the parameter is $c$, the number of vertices (resp. edges) in a $c$-closed graph (or its complement) is denoted by $n$ (resp. $m$), and the number of vertices in a generic graph is denoted by $\numOfV$. 

We state the main theorem of Fox et al. proving that maximal clique enumeration is FPT in $c$.

\begin{theorem}[Fox et al.\cite{DBLP:conf/icalp/FoxRSWW18,fox2020finding}]
\label{thm:cliquesIntro}
	 In any $c$-closed graph, a set of cliques containing all maximal cliques can be generated in time $O(p(n, c) + 3^{c/3}n^2)$, where 
	 $p(n,c) = O(n^{2+o(1)}c +c^{2-\omega-\alpha/(1-\alpha)} n^{\omega} + n^{\omega}\log(n))$ for the matrix multiplication exponent $\omega$ and $\alpha>0.29$.
\end{theorem}
  
\smallskip \noindent \textbf{Complementary terminology} 
We are interested in finding the dense subgraphs in $c$-closed graphs, 
but it is more convenient to present the rest of the paper in the complementary terminology.
This means that we will be working with the complements of $c$-closed graphs. We will use $\bf{m}$ to denote the number of edges in the \textbf{co-graph} (short for complement graph) of a $c$-closed graph. 


\begin{proposition}
A graph	$G$ is the complement of a $c$-closed graph if and only if for any two adjacent vertices $u,v$ in $G$ it holds $|V(G)\setminus N_G[u,v]| \le c-1$. 
\end{proposition}
As the notions of co-treewidth and co-degeneracy are already introduced in the complementary notions, 
it is clear that we are interested in the subgraphs of bounded treewidth and bounded degeneracy in the complement of a $c$-closed graph.

We provide an alternate definition of degenerate graphs, that follows by results of Matula-Beck~\cite{matula1983smallest}.

Given an ordering of vertices $(v_1, \ldots, v_n)$, we will let $V^+(v)$ denote the set of vertices following $v$ in the ordering, and $N^+(v)$ denote the neighbors of $v$ that are after $v$ in the ordering. Thus, $N^+(v) \subseteq V^+(v)$. 
Note that $N^+(v)$ and $V^+(v)$ depend on the ordering, but for brevity we do not it include in the notation as the ordering will always be clear from the context.

\begin{definition} (Degeneracy Ordering) An ordering of vertices $(v_1, \ldots, v_n)$ is a degeneracy ordering if for all $1 \le i \le n$, $v_i$ is the minimum degree vertex in $G[\{v_{i}, \ldots, v_n\}]$, breaking ties lexicographically.
\end{definition}

\begin{definition}($d$-Degenerate Graph)
    A graph $G = (V, E)$ is $d$-degenerate if there exists an ordering $(v_1, \ldots, v_n)$ such that for all $1 \le i \le n$, we have $|N^+(v_i)| \le d$. The degeneracy ordering of a  $d$-degenerate satisfies this property.
\end{definition}
\noindent

We recall that whenever we say maximal subgraph this is referred to a maximal vertex induced subgraph.

\section{Cliques}
\label{section:cliques}

For enumerating cliques, we only consider the backtracking approach, as the three-step approach is already given by Fox et al.~\cite{fox2020finding}.

\begin{definition}(Independent Set Backtracking Tree)
\label{def:backtracking}
    Let $G = (V, E)$ denote the co-graph of a $c$-closed graph and fix an ordering of the vertices. The backtracking tree of $G$ is denoted as $T = (X, F)$ where $X$ is a node-set and $F$ a link-set (we will use nodes and links for the backtracking tree and vertices and edges for $G$).
    A node in $X$ is labeled by a $U \subseteq V$, and a link is labeled by a $v \in V$.
    The tree has the following properties.
    \begin{itemize}
        \item
            The root node is labeled by $V$.
        \item
            All nodes that are labeled by an independent set are leaves.
        \item
            For all internal nodes labeled by $U$, there is a child node for each $v \in U$ labeled by $U' = V^+(v) \setminus N(v)$ with the corresponding link $(U, U')$ labeled by $v$.
    \end{itemize}
    \noindent The root node is at level 0 and the children of any vertex are at exactly one level lower than the vertex.
    We call every $P \cup Q$ an independent set path where $P$ is a root-to-leaf path in $T$ and $Q$ is the last node label of $P$.
\end{definition}

Consider any root-to-leaf path $P = (v_1, \ldots, v_k)$.
By definition of $T$, we have $v_i \in V^+(v_{i - 1}) \setminus N(v_1, \ldots, v_{i - 1})$ for all $1 \le i \le k$.
Hence, $P$ is an induced independent set since $P \subseteq V \setminus N(P)$.
Let the last node label of $P$ be $Q$ which is an independent set since it is a leaf label.
Then, $P \cup Q$ is also an independent set since $Q \subseteq V \setminus N[P]$.
Compiling the above conclusions, it follows that every independent set path in $T$ indeed is an induced independent set in $G$.
Moreover, by the fixed ordering, no two independent set paths correspond to the same independent set.
Now the following converse theorem is fairly straightforward and it does not use the $c$-closure property.
\begin{theorem}
    Every maximal independent set of $G$ is an independent set path in the backtracking tree $T$.
\end{theorem}
\begin{proof}
    Consider a maximal independent set $S$ of size $k$, and let $(v_1, \ldots, v_k)$ be the ordered form of $S$ according to our fixed ordering (in Definition~\ref{def:backtracking}).
    Choose the minimum $j$ such that $V^+(v_j) \setminus N(v_1, \ldots, v_j)$ is an independent set.
    We now show that $P = (v_1, \ldots, v_j)$ is a root-to-leaf path in $T$ and that $Q = \{v_{j + 1}, \ldots, v_k\}$ is the last node label of $P$; hence, $S$ is an independent set path of $T$.
    Further observe that if $P$ is a path starting at the root (a root-originating path), its last node must be a leaf by our choice of $j$.

    We prove that $P$ is a root-originating path by induction on $j$.
    For $j = 0$, this is vacuously true, and for $j = 1$, the claim holds since $v_1 \in V$.
    Now, consider some $j \ge 2$ and assume the inductive hypothesis for $j - 1$, so $(v_1, \ldots, v_{j - 1})$ is a root-originating path.
    Since $v_j \in V \setminus N(v_1, \ldots, v_{j - 1})$, since $S$ is an independent set, and since $v_j$ is of higher order than the vertices $v_1, \ldots, v_{j - 1}$, we have $v_j \in V^+(v_{j - 1}) \setminus N(v_1, \ldots, v_{j - 1})$.
    Thus, by definition of $T$, the path $P$ exists and is a root-originating path.

    Next, since $P$ is a root-to-leaf path, the last node label of $P$ is $U = V^+(v_j) \setminus N(P)$.
    Since $S$ is an independent set, for all $j < i \le k$, we have $v_i \in U$ since $v_i$ has higher order than any vertex in $P$.
    Further, if there exists a $v \in U \setminus \{v_{j + 1}, \ldots, v_k\}$, we have an independent set $P \cup U$ whose subset is $S$, contradicting the maximality of $S$.
    Hence, $Q = U$ as required.
\end{proof}

The key argument that bounds the size of the backtracking tree follows. 
It shows a surprising connection with the $c$-closure parameter.

\begin{lemma}
The backtracking tree $T$ has at most $c$ levels.
\end{lemma}
\begin{proof}
We show the lemma by showing that for every independent set $S$ of size $c$, the set of its non-neighbours is also an independent set. 

Let $U = V \setminus N[S]$ be the set of non-neighbours of $S$. 
We claim $U$ is an independent set.
If $G[U]$ were to contain an edge $\{u, v\}$, then $S \subseteq V \setminus N[u, v]$ since $S \cup \{u\}$ and $S \cup \{v\}$ are independent sets.
Since $|S| = c$, we breach the $c$-closed condition; thus, $U$ must be an independent set.
Hence $T$ has at most $c$ levels, since every node at level $c$ is a leaf node. 
\end{proof}

\begin{theorem}\label{thm:cliques}
   The size of the backtracking tree $T$ is $O(cm2^c)$.
\end{theorem}
\begin{proof}
    For any non-leaf node label $U$, the induced subgraph $G[U]$ contains an edge.
    For any edge $e = \{u,v\}$, let us count the number of such tree nodes such that $G[U]$ contains $e$.
    Let $P$ be the path in $T$ from the root to $U$.
    Then we have $P \subseteq V \setminus N[u, v]$ since $P \cup \{u\}$ and $P \cup \{v\}$ are independent sets.
    Since $|V \setminus N[u, v]| < c$ and all paths are unique, the edge $e$ can appear in at most $\binom{c}{i}$ non-leaf nodes at level $i$. In other words, the number of occurrences of edge $e$ at level $i$ can be at most $\binom{c}{i}$. Thus the total number of occurrences of all edges at level $i$ is at most $\sum\limits_{e \in E(G)}\binom{c}{i}=m\binom{c}{i}$. In other words, if we let $\bU_i$ be the set of all non-leaf nodes at level $i$, then $\sum\limits_{U \in \bU_i}|E(U)| \leq m\binom{c}{i}$. Note that this means that $|\bU_i| \leq m\binom{c}{i}$.

    The number of isolated vertices in $G[U]$ is less than $c$ since $G[U]$ contains an edge, and the number of non-isolated vertices in $G[U]$ is at most $2|E(U)|$.
    Hence, the node labeled by $U$ can have at most $2|E(U)| + c$ children.
    Thus the number of all children produced at level $i$ (i.e. the total number of nodes in the tree $T$ at level $i+1$) is at most
     \begin{align*}
        \sum\limits_{U \in {\bU_i}}(2|E(U)|+c) \leq 2\sum\limits_{U \in {\bU_i}}(|E(U)|) + c|\bU_i| \leq 2m{c \choose i} + cm{c \choose i} = (2+c)m{c \choose i},
    \end{align*}

    Thus, the total number of nodes in $T$ is given by
    \begin{align*}
        (2+c)m\sum\limits_{i = 0}^c {c \choose i} = O( cm2^c)
    \end{align*}
    as desired.\qedhere
\end{proof}
\noindent To construct the children for every internal node of this tree will take $O(n^2)$ time, so to build $T$ and enumerate a superset of maximal independent sets in $G$ (equivalently, a superset of maximal cliques in the $c$-closed graph whose complement is $G$) will take $O(cmn^22^c)$ time. Thus, the backtracking algorithm runs in FPT time with parameter $c$.
Interestingly, the backtracking algorithm does not need to know the value of the parameter $c$. 

\begin{corollary}\label{cor:backtrackingEnum}
The backtracking algorithm enumerates a superset of all maximal independent sets in the co-graph of a $c$-closed graph in time $O(cmn^22^c)$, where $m$ is the number of edges in the co-graph and $n$ is the number of vertices.
\end{corollary}

\section{$(d+1)$-plexes}
\label{section:plexes}


For any fixed $d$, we show that the number of maximal subgraphs with degree at most $d$ in the complement of a $c$-closed graph 
admits an FPT bound.
This implies that the number of maximal $(d+1)$-plexes in a $c$-closed
graph admits an FPT bound and an FPT enumeration algorithm. 

We give proofs using both approaches, starting with the approach that uses backtracking as a subroutine.

\begin{theorem}
Let $G$ be the complement of a $c$-closed graph. The number of maximal subgraphs with degree at most $d$ in $G$ is bounded by $O(mn^{2d-2}2^c)$.
\end{theorem}
\begin{proof}
We count two types of maximal subsets $S$ that induce a subgraph with degree at most $d$:
\begin{itemize}
 	\item subsets $S$ for which $G[S]$ is edgeless, and 
 	\item subsets $S$ for which $G[S]$ contains at least one edge. 
 \end{itemize} 
If $G[S]$ is a maximal subgraph with degree at most $d$ and $G[S]$ is edgeless, 
then $S$ is also a maximal independent set in $G$.
By Corollary~\ref{cor:backtrackingEnum}, a superset of all maximal independent sets in $G$ can be enumerate in time $O(cmn^22^c)$.

Suppose $G[S]$ has an edge, say $(u,v)$. Let $Y=S \cap N(u,v)$ and $Z=S \setminus N[u,v]$, then by the $c$-closed condition, $|Z| \leq c$. Moreover, since $Y$ consists of neighbors of $(u,v)$ and $u$ and $v$ can have at most $d-1$ neighbors, $|Y| \leq 2d-2$. For any edge, there are $2^c$ possible choices for $Z$ and $O(n^{2d-2})$ choices for $Y$. Hence, the number of maximal $(d+1)$-plexes containing at least one edge is $O(mn^{2d-2}2^c)$. By simply enumerating all possible choices for $Y$ and $Z$ for every edge and combining them, in total time $O(mn^{2d-2}2^c)$, we will have enumerated a superset of all $(d+1)$-plexes containing an edge.
\end{proof}

\begin{corollary}
Let $G$ be the complement of a $c$-closed graph. A superset of all maximal subgraphs with degree at most $d$ in $G$ can be enumerated in time $O(mn^{2d-2}2^c + cmn^22^c)$.
\end{corollary}

\subsection{Enumerating $(d+1)$-plexes via the three-step approach}
Next, we give an alternate bound with exponential improvement in $c$ is using the three step approach. The running time bound we obtain is $O(n^{2d} \cdot \kappa_d^{c} \cdot p(c))$ where $\kappa_d < 2$ is the root of $x^{d+4} - 2x^{d+3} + 1=0$; and for a polynomial $p$.

\smallskip \noindent \textbf{Combinatorial bound} 
Our bound depends on an extension of $\m_d(\numOfV)$.
For a (not necessarily $c$-closed) graph $G$ and $P\subseteq V(G)$, the number of maximal subgraphs containing $P$ and with degree at most $d$ is denoted by  $\m_d(G; P)$.
Analogously, $\m_d(\numOfV+p; p)$ is the maximum value $\m_d(G; P)$ takes over all graphs on $\numOfV+p$ vertices and all sets $P\subseteq V(G)$ with size $p$.
In particular, $\m_d(\numOfV) = \m_d(\numOfV; 0)$. 
By adding isolated vertices, it is easy to see that $\m_d(\numOfV+p; p) \le \m_d(\numOfV+p'; p')$ for all $p\le p'$.

By closely examining the result by Zhou et al.~\cite[Theorem 1]{zhou2020enumerating}, 
we note that they implicitly show that for each $d$ and every $p$
there is a constant $\kappa_d<2$ such that $\m_d(\numOfV+p; p) \le \kappa_d^{\numOfV}$.
More precisely, they show that the bound holds if $\kappa_d$ is the positive solution of $x^{d+3} - 2x^{d+2} + 1=0$.
For $d= 0,\dots, 4$ we have $\kappa_d = 1.618, 1.839, 1.928, 1.966$ and $1.984$.
To the best of our knowledge, next to the Moon-Moser theorem, 
these are the best (and only) existing bounds for $\m_d(\numOfV)$ and $\m_d(\numOfV + p; p)$.

The Moon-Moser theorem  states that $\kappa_0 = 3^{1/3}$ suffices.  
In Appendix~\ref{section:combinatorics}, we prove a tight upper bound on $\m_1(\numOfV)$.
In other words we show that we can set $\kappa_1 = 10^{1/5} \le 1.585$. 
The proof uses similar recursive bound(s) as in the Moon-Moser theorem (Theorem~\ref{thm:moonMozer}), 
and in the proof for $1$-regular graphs given by Gupta et al.~\cite[Theorem 4]{rRegular}, 
but our proof requires a significantly more extensive case analysis.

\begin{restatable}{theorem}{generalizedInducedMatching}
\label{thm:generalizedInducedMatching}
  $\m_1(\numOfV) \le 10^{\numOfV/5} \le 1.585^{\numOfV}$.
\end{restatable}

To see that the bound is tight consider any $\numOfV$ a multiple of $5$. 
The graph consisting of $\frac{\numOfV}{5}$ copies of $K_5$ 
contains $10^{\numOfV/5}$ maximal subgraphs with degree at most 1.
The same number of subgraphs is attained if we remove a matching from each of the $K_5$s.

\smallskip \noindent \textbf{FPT bound}
Our next goal is to give an upper bound on the number of subgraphs with degree at most $d$ in the complement of a $c$-closed graph using 
$\m_d(\numOfV +p; p)$ for $d>1$, and $\m_1(\numOfV)$. 
For the case when $d=0$, we already have Theorem~\ref{thm:cliquesIntro} which we use in the proof. 

\begin{theorem}
	\label{thm:boundedDegree}
	Let $G$ be the complement of a $c$-closed graph.
	The number of maximal induced subgraphs with degree at most $d$ in $G$,
	is bounded by $2 n^{2d} \cdot \m_d(c - 1 + 2d; 2d)$.
	Moreover, for $d=1$ the bound simplifies to $2 n^2 \cdot \m_1(c - 1)$.
\end{theorem}
\begin{proof}
Similar to the proof for the first bound for counting $(d+1)$-plexes, we count two types of maximal subsets $S$ that induce a subgraph with degree at most $d$:
\begin{itemize}
 	\item subsets $S$ for which $G[S]$ is edgeless, and 
 	\item subsets $S$ for which $G[S]$ contains at least one edge. 
 \end{itemize} 
If $G[S]$ is a maximal subgraph with degree at most $d$ and $G[S]$ is edgeless, 
then $S$ is also a maximal independent set in $G$.
By Theorem~\ref{thm:cliquesIntro}, the number of maximal independent sets in $G$ 
is bounded by $n^2 \cdot  \m_0(c-1)$.
By definition, it is not hard to see that $\m_0(c-1) \le \m_0(c-1+d; d) \le \m_d(c-1+2d; 2d)$ holds.
Therefore, in order to prove the theorem, it suffices to show that 
the number of maximal subgraphs that contain an edge and with degree at most $d$
is bounded by $n^{2d} \cdot \m_d(c - 1+2d; 2d)$.

\begin{figure}[h]
	\centering
	\includegraphics[width = 0.7 \textwidth]{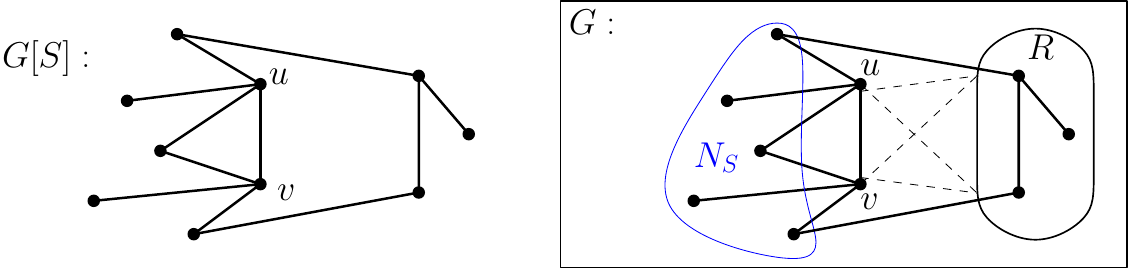}
	\caption{Proof of Theorem~\ref{thm:boundedDegree}. 
	Left: $G[S]$ represents an induced subgraph with maximum degree $4$.
	Right: depiction of $G[S]$ within $G$. 
	Recall that $R= V(G)\setminus N_G[u,v]$ and that $|R|\le c-1$ as $G$ is the complement of a $c$-closed graph.
	The dashed lines represent non-edges.}
	\label{figure:kplexProof}
\end{figure}

We refer to Figure~\ref{figure:kplexProof}. 
Let $uv$ be an edge in $G$. 
Suppose that $S$ is a maximal set such that $\Delta(G[S]) \le d$ and $u,v \in S$.   
Let $N_S = S\cap N_G(u,v)$.
By the maximum degree assumption and since $u$ and $v$ are adjacent to each other, 
there are at most $2d-2$ vertices in $N_S$.
To prove the theorem, we show that the number of maximal sets $S$ satisfying the following two 
\begin{itemize}
	\item degree of $G[S]$ is at most $d$, and
	\item $S$ contains $\{u,v \}$ and $S\cap N(u,v) = N_S$ ($S$ contains $\le 2d$ fixed vertices);
\end{itemize}
is bounded by $\m_d(c - 1+2d; 2d)$.

\smallskip
We claim that any such maximal set $S$ also induces a maximal subgraph (with the same properties) 
in graph $G[\{u,v\} \cup N_S \cup R]$ where $R = V\setminus N_G[u,v]$.
Namely, we can obtain $G[\{u,v\} \cup N_S \cup R]$ from $G$ by removing some vertices that are not in $S$.
As removal of such vertices does not influence the maximality of $S$, it follows that 
$S$ induces a maximal subgraph (with the above stated properties) in $G[\{u,v\} \cup N_S \cup R]$.

Since $G$ is the complement of a $c$-closed graph and by definition of $R$, we have $|R| \le c-1$.
Let $k = |\{u,v\} \cup N_S|$.
Then, by definition of $\m_d(c -1+k; k)$ it follows that 
the number of maximal sets $S$ that induce a subgraph with degree at most $d$ and contain $\{u,v\} \cup N_S$
is bounded by $\m_d(c - 1+k; k)$.
As $|\{u,v\} \cup N_S| = k\le 2d$ we have $\m_d(c-1+k; k) \le \m_d(c-1+2d; 2d)$ and the proof follows.

\smallskip
Next, we deal with the case $d=1$.
The proof is largely the same and we make a small change in the way we count the subsets $S$ that contain $u, v$. 
As the maximum degree of $G[S]$ is at most $1$ and since $u$ and $v$ are adjacent to each other 
we have that $N_S = \emptyset$. 
We claim that if $S$ is maximal set with degree at most $1$ in $G$ containing $uv$,
then $S\setminus \{u,v\}$ is a maximal set with degree at most $1$ in $G[R]$.

For a contradiction, suppose that $S\setminus \{u,v\}$ is not a maximal such set, 
and let $S' \subseteq R$ such that $S\setminus \{u,v \}\subset S'$ and $\Delta(G[S']) \le 1$.
Since $S' \subseteq R = V(G) \setminus N[u,v]$ it follows that $u$ and $v$ are non-adjacent to $S'$.
Thus, $\Delta(G[S' \cup \{u,v\}]) \le 1$ contradicting maximality of $S$.

It follows that the number of maximal subsets $S$ with $\Delta(G[S])\le 1$ 
and that contain edge $uv$ is at most $\m_1(c-1)$.
Thus, the number of maximal subsets $S$ with $\Delta(G[S]) \le 1$ is bounded by $n^2\m_0(c-1) + n^2\m_1(c-1) \le 2n^2\m_1(c-1)$.
\end{proof}

We give an example showing that the dependency on $n$ and $d$ cannot be improved.
\begin{example}\label{examplePlexes}
	Any complete bipartite graph is the complement of a $1$-closed graph 
	as any two adjacent vertices have no common non-neighbors.
	Let $K_{i,j}$ be the complete bipartite graph with parts of size $i$ and $j$.
	It is easy to see that the number of maximal subgraphs with degree at most $d$ in $K_{\ell, \ell}$ for $\ell >d$, 
	is at least $\Omega(\ell^{2d}) = \Omega\left( \frac{|V(K_{\ell, \ell})|^{2d}}{2^{2d}} \right)$ for any fixed $d$.
 \end{example}

\smallskip \noindent \textbf{Enumeration}
Equipped with Theorem~\ref{thm:boundedDegree} it is straightforward to obtain an algorithm, with running time 
similar to the FPT bound, for enumeration of all maximal $(d+1)$-plexes in $c$-closed graph. 
A simple way is to run a polynomial delay algorithm for listing 
all maximal subgraphs with degree at most $d$ on the complement graph~\cite{berlowitz2015efficient}. 
The FPT bound then implies that the enumeration algorithm indeed runs in FPT time. 
A better running time can be obtained if the enumeration algorithm is incorporated directly into the proof of the FPT bound. 
We sketch it below.

\begin{restatable}{corollary}{corPlexes}\label{cor:fasterBoundedDegree}[Restatement of Theorem \ref{thm:fasterBoundedDegree}]
	For $c$-closed graphs and a fixed $d \ge 0$, 
	there is an algorithm running in time $O(n^{2d} \cdot \kappa_d^{c} \cdot p(c))$ for enumerating $(d+1)$-plexes, 
	where $\kappa_d < 2$ is the root of $x^{d+4} - 2x^{d+3} + 1=0$; and for a polynomial $p$.
 	For $2$-plexes, a stronger bound $O(n^{2} \cdot 10^{c/5} \cdot p(c))$ applies.
\end{restatable} 

\begin{proof}[Proof of Corollary~\ref{cor:fasterBoundedDegree}]
	We enumerate all maximal subgraphs with degree at most $d$ in the complement graph. 
	If a maximal subgraph with degree at most $d$ is edgeless,
	then it is also a maximal independent set and we use the algorithm by Fox et al.~\cite{fox2020finding} stated in Theorem~\ref{thm:cliquesIntro}.	 
	
	Hence, we only need to enumerate the maximal subgraphs with degree at most $d$ and that contain at least one edge. 
	Similarly, as in the proof of Theorem~\ref{thm:boundedDegree} once we fix an edge $uv$, and the neighbors of $u$ and $v$
	the rest of maximal induced subgraph is contained in a subset of at most $c-1$ vertices. 
	By applying the polynomial delay algorithm~\cite{berlowitz2015efficient} to these vertices, 
	we can obtain all maximal subgraphs of degree at most $d$ that contain the fixed vertices in time 
	$O(\m_d(c - 1 + 2d; 2d) \cdot p(c)) \le  O(\kappa_d^{c} \cdot p(c))$ for a polynomial $p$.
\end{proof}

\section{Bounded co-degeneracy}
\label{section:degeneracy}

As with $(d+1)$-plexes, we first give the result with the backtracking approach.

Any $d$-degenerate graph (with possible isolated vertices) can either be an independent set or it can be separated into 3 components, characterized by an edge in the graph.
This decomposition is unrelated to the $c$-closed property, but we exploit this structure for faster enumeration in a $c$-closed co-graph.

\begin{lemma}
    Consider a $d$-degenerate graph $H$ with the degeneracy ordering of $(u_1, \ldots, u_n)$.
    If $H$ is not an independent set, there exists an edge $(u_s, u_t)$ such that for
    \begin{align*}
        X = \{u_1, \ldots, u_{s - 1}\}, && Y = \{u_{s + 1}, \ldots, u_{t - 1}\}, && Z = \{u_{t + 1}, \ldots, u_n\},
    \end{align*}
    $X$ is an independent set, $Y$ is a subset of $V \setminus N[u_s, u_t]$, and $Z$ is a subset $V \setminus N[u_s, u_t]$ with at most $2d-2$ additional vertices.
\end{lemma}
\begin{proof}
    Choose minimum $t$ such that $u_t$ is a terminal vertex of an edge in $H$.
    Then choose maximum $s$ such that $(u_s, u_t)$ is an edge in $H$ (this must exist since $H$ is not an independent set).
    By the minimality of $t$, $X$ is an independent set.

    By the minimality of $t$, $u_s$ is not adjacent to any vertex in $Y$.
    By the maximality of $s$, $u_t$ is not adjacent to any vertex in $Y$.
    Hence, $Y \subseteq V \setminus N[u_s, u_t]$.
    
    Furthermore, since $u_t$ and $u_s$ are connected, each can be adjacent to at most $d-1$ vertices in $Z$ to ensure the $d$-degeneracy condition.
    Thus, the rest of the vertices in $Z$ are non-adjacent from both $u_t$ and $u_s$.
\end{proof}
\noindent Notice that since $|V \setminus N[u_s, u_t]| < c$ by the $c$-closed condition, we have $|Y| < c$ and $|Z| < 2d - 2 + c$. 
Furthermore, note that if $H$ is maximal, then so is the independent set $X$.

\begin{restatable}{theorem}{thmDegeneracy}[Restatement of Theorem \ref{thm:max-degen}]
 For $c$-closed graphs and a fixed $d \ge 0$, 
	there is an algorithm running in time $O(cm^2n^{2d}4^c + cmn^22^c)$ that outputs a set containing
all maximal induced  subgraphs with co-degeneracy $d$ in the $c$-closed graph, where $m$ is the number of edges in the complement graph of the $c$-closed graph.
\end{restatable}

\begin{proof}
We describe an algorithm that generates supersets of all maximal induced $d$-degenerate subgraphs in a $c$-closed co-graph $G$.
(We can check in linear time whether each such subgraph is truly $d$-degenerate.)

Start with any edge $\{u, v\}$ and pick an orientation (say) $(u,v)$.
Then, we construct all possible choices of $Y$ and $Z$, which takes $O(n^{2d-2}2^c)$ time.
Next, we choose $Y$ and $Z$ such that $G[Y \cup Z \cup \{u,v\}]$ is a $d$-degenerate subgraph whose degeneracy ordering is $(u, Y, v, Z)$.
Then, we can build a set $S$ of vertices $s$ where $s$ is the first vertex in the degeneracy ordering of $G[\{s, u, v\} \cup Y \cup Z]$.
Lastly, we enumerate all maximal independent sets $X$ in $G[S]$ which takes $O(cmn^22^c)$ time by Corollary~\ref{cor:backtrackingEnum}.
Then, any maximal $d$-degenerate subgraph $H$ of $G$ is $X \cup Y \cup Z \cup \{u, v\}$ for some chosen $X$, $Y$, $Z$, and $\{u, v\}$ according to the above algorithm.
The total run-time for this algorithm is $O(cm^2n^{2d}4^c + cmn^22^c)$.
\end{proof}

\subsection{
Enumerating subgraphs of bounded co-degeneracy with the three-step approach}
We give another FPT algorithm for enumerating all maximal subgraphs with degeneracy at most $d$ in the complement of a $c$-closed graph using the three-step approach. For this (as well as for bounded-treewidth) we use the notion of a generalized star and of an $(\ell, k)$-partition. We define these below. 
The bound obtained using this approach is worse than the algorithm described above but we include it for the sake of completeness and since the same notions are used in the case of bounded treewidth.
The proof uses an alternate characterization of the structure of a bounded-degeneracy graph in the co-graph of a $c$-closed graph. 
For details and missing proofs we refer to Appendix~\ref{section:degeneracyAppendix}.

\noindent \textbf{Generalized stars} We say that that a graph $H$ is a $k$-star if there is a partition $\{A, B\}$  of $V(H)$
such that $|A|\le k$ and $B$ is an independent set.
Equivalently, graph is a $k$-star if and only if it has a vertex cover of size at most $k$.
We say that $A$ is the \emph{head} of the $k$-star $H$, and $B$ is the set of \emph{tails}. 
A $k$-star is \emph{proper} if every tail is adjacent to at most $k-1$ vertices (in the head).
In particular, any $(k-1)$-star is a proper $k$-star.
We note that an edgeless graph is a proper $1$-star and a (vertex disjoint) union of an edgeless graph and a star is a proper $2$-star.

\begin{lemma}\label{lemma:numberOfKstars}
	Let $G$ be the complement of a $c$-closed graph.
	The number of subsets $S\subseteq V(G)$ that induce a proper $k$-star with a 
	maximal set of tails is at most $2 \cdot n^{k+2} \cdot \m_0(c-1)$.
\end{lemma}
Note that we only require that the set of tails is maximal: 
there is no proper $k$-star with the same head and a strictly larger (inclusion-wise) set of tails.

\begin{proof}[Proof of Lemma~\ref{lemma:numberOfKstars}]
	Let $A\subseteq V(G)$ be a set of at most $k$ vertices. 
	For a proper $k$-star with head $A$ and the set of tails $B$ 
	it holds that $B$ is an independent set in $G\setminus A$.
	Suppose that the $B$ is the maximal set of tails for the $k$-star $G[A\cup B]$.
	
	Let $X$ be the set of vertices $v \in V(G)\setminus A$ that are adjacent to every vertex in $A$.
	If $|A| = k$, then since $G[A\cup B]$ is proper and by maximality of the tail, 
	it follows that $B$ is a maximal independent set in $G\setminus (A\cup X)$.
	If $|A| < k$ then by the maximality of tail, $B$ is a maximal independent set in $G\setminus A$.

	By Theorem~\ref{thm:cliquesIntro} there are at most $n^2 \m_0(c-1)$ maximal independent sets in $G\setminus A$ 
	and similarly at most $n^2 \m_0(c-1)$ maximal independent sets in $G\setminus (A\cup X)$.
	The lemma follows. 
\end{proof}

\noindent \textbf{Good $(\ell, k)$-partitions} Next, we introduce a definition that captures the property of graphs we can count by fixing several edges.
Informally, we say that a graph $H$ admits a good $(\ell, k)$-partition if there are $k$ edges and a set $A_0$ on at most $\ell$ vertices 
such that  the rest of the graph can be partitioned into non-neighborhoods of the edges. 
We show that the subgraphs admitting a good $(\ell, k)$-partition are easy to count.

\begin{definition}
	We say that a graph $H$ admits a \emph{good $(\ell, k)$-partition} if there exist $k$ edges $e_1, \dots, e_k$
	and a $(k+1)$-partition $\{A_0, A_1, \dots, A_k\}$ of the set $V(H)\setminus \left(\cup_{i=1}^{k} e_i\right)$ 
	such that $N_H(e_i) \cap A_i = \emptyset$ for every $i\in [k]$ and $|A_0|\le \ell$.
\end{definition}

\begin{lemma}\label{lemma:kEdges}
	Let $G$ be the complement of a $c$-closed graph. 
	The number of subsets $S \subseteq V$ for which graph
	$G[S]$ admits a good $(\ell, k)$-partition,  is bounded by $n^{\ell+2k}\cdot 2^{k(c-1)}$.
\end{lemma}

\begin{proof}[Proof of Lemma~\ref{lemma:kEdges}]
	Let $H$ be induced subgraph of $G$ that let $e_1, \dots, e_k$  and $A_0, \dots, A_k$ 
	be the edges and sets defining a good $(\ell, k)$-partition of $H$.
	To prove the lemma, it suffices to show that the number of induced subgraphs that admit a good $(\ell, k)$-partition 
	with the same edges $e_1, \dots, e_k$ and the same set $A_0$ is bounded by $2^{k(c-1)}$.

	Denote with $U$ the vertices of $G$ that are neither incident to the edges $e_1, \dots, e_k$ nor in the set $A_0$, 
	i.e., $U = V(G)\setminus (A_0 \cup_{i=1}^k e_i)$. 
	By definition of a good $(\ell, k)$-partition, for any induced subgraph with a good $(\ell, k)$-partition 
	$e_1, \dots, e_k$ and $A_0, A'_1, \dots A'_k$ it holds $A'_i\subseteq U \setminus N_G(e_i)$ for each $i \in [k]$.
	Since $G$ is complement of a $c$-closed graph, it follows that $|U\setminus N(e_i)| \le c-1$ for each $i\in[k]$. 
	Hence, there are at most $2^{k(c-1)}$ induced subgraphs $G[S]$ that admit a good $(\ell, k)$-partition with $A_0$ and the edges $e_1, \dots, e_k$.
	The lemma follows.
\end{proof}
.

We obtain an FPT algorithm for bounded-degeneracy graphs in the following way.

\smallskip 
\noindent \textbf{Combinatorial bound}
Recall that the maximum number of maximal $d$-degenerate subgraph with in an arbitrary $\numOfV$-vertex graph is denoted by 
$\dd_d(\numOfV)$.
Pilipczuk and Pilipczuk~\cite{pilipczuk2012finding} show that for every $d$ there is a constant $\gamma_d<2$ such that $\dd_d(\numOfV) \le \gamma_d^{\numOfV}$.

\smallskip 
\noindent \textbf{FPT bound}
It can be shown that a $d$-degenerate graph is either a $4d$-star or admits a good $(4d, 2d)$-partition.
Then, by Lemmas~\ref{lemma:numberOfKstars} and~\ref{lemma:kEdges} we obtain an FPT upper bound.
\begin{restatable}{theorem}{FPTboundeDegeneracy}\label{thm:averageDegree}
	Let $G$ be the complement of a $c$-closed graph.
	The number of maximal $d$-degenerate subgraphs in $G$ is bounded by $O(n^{8d} \dd_d(2dc))$.
\end{restatable}

\noindent \textbf{Enumeration}
Maximal $d$-degenerate subgraphs can be listed in time $O(mn^{d+2})$ per  maximal subgraph~\cite{conte2019proximity}.
We obtain the following corollary.
\begin{restatable}{corollary}{enumerationDegeneracy}
	For each fixed integer $d$, there is a constant $\gamma_d < 2$ and an
FPT algorithm running
 in time $O(n^{9d+4} \cdot \gamma_d^{2dc})$ for enumerating all maximal subgraphs with co-degeneracy at most $d$ in a $c$-closed graph $G$.
\end{restatable}

\section{Bounded co-treewidth}
\label{section:treewidth}
We give FPT algorithms for enumerating all maximal subgraphs of bounded treewidth in the complement of a $c$-closed graph using (only) the three-step approach.
For the combinatorial bound, 
we use the trivial upper bound $2^{\numOfV}$ for the number of maximal subgraphs of bounded treewidth in an $\numOfV$-vertex graph. 
For the enumeration, we are unaware of any polynomial delay algorithms for enumerating maximal subgraphs of bounded treewidth. 
Nevertheless, the proof of the FPT bound is easily turned into an FPT enumeration algorithm.
Therefore, we are only concerned with proving the FPT bound. 
In Appendix~\ref{section:localTreeWidth}, we extend the upper bound (and consequently the algorithm) to the subgraphs of bounded local treewidth. 

\smallskip \noindent \textbf{FPT bound}
To count star-like maximal subgraphs with treewidth at most $t$ in the complement of a $c$-closed graph, we use Lemma~\ref{lemma:numberOfKstars}.
The counting reduces to counting maximal independent sets in smaller graphs.

To count the non-star-like graphs with treewidth at most $t$, 
we use Lemma~\ref{lemma:kEdges}. The lemma shows how to count all subgraphs that contain several edges and show that any other vertex is non-adjacent to at least one of the fixed edges. 

The upper bound is proved by combining the two mentioned cases. 
More precisely, we show that any subgraph of bounded treewidth is counted by either Lemma~\ref{lemma:numberOfKstars} or Lemma~\ref{lemma:kEdges}.

\smallskip

We present the main theorem of this section. 

\begin{theorem}\label{thm:boundedTW}
	Let $G$ be the complement of a $c$-closed graph and let $t\in \N$. 
	The number of maximal subsets $S\subseteq V(G)$ for which $\tw(G[S]) \le t$ is at most $3n^{t+4} 2^{2(c-1)}$.	
\end{theorem}

Before we prove the theorem, we mention that the class of all graphs with treewidth at most $t$ contains all ``proper'' $(t+1)$-stars but not all $(t+1)$-stars. Simply, $K_{t+2}$ is a $(t+1)$-star but has treewidth $t+1$. The proof relies on the following claim.

\begin{claim}\label{claim:separatorOrStar}
	Let $S\subseteq V(G)$ such that $\tw(G[S]) \le t$. 
	Then, $G[S]$ is either a proper $(t+1)$-star or admits a good $(t, 2)$-partition. 
\end{claim}
\begin{claimproof}[Proof of Claim~\ref{claim:separatorOrStar}]
	Let $(T, \W)$ be a tree decomposition of $G[S]$ of width at most $t$;  
	$W_a$ is the bag corresponding to vertex $a \in V(T)$ and $\W$ is the set of bags, i.e., $\W= \{W_a : a\in V(T)\}$.	
	Without loss of generality, we may assume that for any edge $ab \in E(T)$ the bags $W_a$ and $W_b$ are crossing,
	i.e., it holds $W_a \setminus W_b \neq \emptyset \neq W_b \setminus W_a$. 
	On the contrary, if $W_a \subseteq W_b$ we can simply remove the vertex $a$ and the bag $W_a$ and 
	reconnect the tree in the natural way to obtain a tree decomposition with the same width and a smaller tree.

	Let $ab\in E(T)$ and let $T_a, T_b$ be the trees in $T\setminus ab$.
	Tree $T_a$ (resp. $T_b$) is the tree in $T\setminus ab$ containing the vertex $a$ (resp. $b$).
	It is easy to check that there is no edge between $U_a := \cup_{t\in V(T_a)} W_t \setminus (W_a \cap W_b)$ 
	and $U_b := \cup_{t\in V(T_b)} W_t \setminus (W_a \cap W_b)$. 
	In other words, $W_a \cap W_b$ is a separator of $G[S]$ whenever $U_a \neq \emptyset \neq U_b$.
	Since the adjacent bags in $T$ are crossing we do have $U_a \neq \emptyset \neq U_b$.
	Moreover, since $|W_a|, |W_b|\le t+1$ and $W_a\setminus W_b \neq \emptyset$ it follows that $|W_a\cap W_b| \le t$.
	Thus, $W_a\cap W_b$ is a separator of size at most $t$ in $G[S]$ for every $ab \in E(T)$.
	If $U_a$ and $U_b$ both contain an edge, say $e_1$ and $e_2$ respectively, then $G[S]$ admits a good $(t, 2)$-partition.
	Namely, we can set $A_0 = W_a \cap W_b$, $A_1 = U_b$, and  $A_2 =  U_a$.
	Therefore, we assume that for each edge $ab \in E(T)$ at least one of $U_a$ or $U_b$ is an independent set.
	We show, that this implies that $G[S]$ is a proper $(t+1)$-star.
		
	\smallskip
	If for some $ab \in E(T)$ both $U_a$ and $U_b$ are independent sets, then so is $U_a \cup U_b$. 
	As $U_a \cup U_b = S \setminus (W_a \cap W_b)$ and $|W_a\cap W_b|\le t$,  it follows that $G[S]$ is a $t$-star.
	Hence, for the rest of the proof we assume that for each edge $ab \in E(T)$ either $U_a$ or $U_b$ is not an independent set. 
	Combining with the previous paragraph, we have that for each $ab \in E(T)$ exactly one of 
	$U_a$, $U_b$ is an independent set and the other one is not. 
	
	Such a property gives a natural orientation of the edges in $T$. 
	In particular, if $U_a$ is an independent set we orient the edge $ab$ as $(a,b)$ and say that edge $ab$ is oriented \emph{towards} $b$. Otherwise we orient $ab$ as $(b,a)$ as say that $ab$ is oriented \emph{towards} $a$.
	Since $T$ is a tree, there is a vertex $s \in V(T)$ such that all incident edges are oriented towards $s$. 
	(Start with an arbitrary vertex $x\in V(T)$ and move to any vertex $y \in N_T(x)$ such that $xy$ is oriented towards $y$. We keep iterating until we encounter a vertex $s$ such that all incident edges are oriented towards $s$. The process terminates as $T$ is a tree.)
	We show that $S \setminus W_s$ is an independent set. 

	Suppose on the contrary that there is an edge $uv \in G[S]\setminus W_s$. 
	By the definition of tree decomposition $(T, \W)$, 
	the vertices $u$ and $v$ are both contained in some bag $W_p$ for $p \in V(T)$.
	Moreover, it holds that $p \neq s$. 
	Let $q$ be the neighbor of $s$ on the undirected $s$-$p$ path in $T$ (possibly $q = p$). 
	Then, $U_{q}$ is not an independent set: we have $uv \in U_{q}$ since $W_p\setminus W_s \subseteq U_{q}$.
	It follows that the edge $sq$ is oriented from $s$ to $q$. A contradiction with the choice of $s$. 
	As $|W_s| \le t+1$ we conclude that $S$ is a $(t+1)$-star.

	\smallskip
	 It remains to show that the $(t+1)$-star is proper, i.e.,
	 that every vertex $v\in S\setminus W_s$ is adjacent to at most $t$ vertices in $W_s$.
	 If $|W_s| \le t$, then there is nothing to prove, so assume $|W_s| = t+1$.
	 For the sake of contradiction, let $v \in S\setminus W_s$ be a vertex adjacent to all $t+1$ vertices of $W_s$. 
	 Let $W_r$ be the bag containing $v$ that is closest to the bag $W_s$ in the tree $T$. 
	 Let $T_v$ be the tree in $T\setminus s$ that contains $r$.
	 Since $v\not \in W_s$, for any bag $W_x$ that contains $v$ it holds $x\in T_v$.
	 Moreover, the unique $s-x$ path in $T$ contains the vertex $r$.
	 By the properties of tree decomposition, and since $v$ is adjacent to every vertex in $W_s$ it follows that $W_s \subset W_r$.
	 Thus, $|W_r| \ge |W_s \cup \{v\}| = t+2$. 
	 A contradiction with the width of $(T, \W)$. 
\end{claimproof}

\begin{proof}[Proof of Theorem~\ref{thm:boundedTW}]
	Let $S$ be a maximal subset of vertices of $G$ such that $\tw(G[S]) \le t$.
	By Claim~\ref{claim:separatorOrStar}, either $G[S]$  admits a good $(t, 2)$-partition or
	$S$ induces a proper $(t+1)$-star. 
	The number of sets $S$ that admit a good $(t,2)$-partition is at most $n^{t+4}2^{2c-2}$ by Lemma~\ref{lemma:kEdges}.

	Let us consider the case when $G[S]$ is a proper $(t+1)$-star.
	Since $S$ is a maximal set with property that $G[S] \in \CC$ it follows that 
	$S$ is also a set that induces a proper $(t+1)$-star with maximal tail. 
	It is not hard to see that the class of graph with bounded treewidth contains all proper $(t+1)$-stars. 
	The number of sets $S$ that induce a proper $(t+1)$-star with maximal tail is at most $2n^{t+3} 2^{c-1}$ by Lemma~\ref{lemma:numberOfKstars}.
	The theorem follows.
\end{proof}


\begin{example}\label{exampleTW}
	Recall that $K_{a,b}$ is the complement of a $1$-closed graph, and that $\tw(K_{a, b}) = \min \{a,b\}$ for any $a,b \in \N$.
	Trivially, $K_{\ell, t+1}$ contains at least $\Omega(\ell^t) = \Omega\left( (|V(K_{\ell, t+1})| - t-1)^t\right)$ maximal induced subgraphs with treewidth at most $t$. Hence, the dependence on $n^t$ in Theorem~\ref{thm:boundedTW} is necessary.
\end{example}

\smallskip \noindent \textbf{Enumeration}
Let us explain how to turn the above proof in an enumeration algorithm. 
In the proof of Theorem~\ref{thm:boundedTW} we showed that
any maximal induced subgraph of treewidth at most $t$ is either a proper $(t+1)$-star or admits a good $(t, 2)$-partition.

Enumeration of all proper $(t+1)$-stars reduces to the enumeration of all maximal independnet sets in the complement of smaller $c$-closed graphs by the same reduction as in the proof of Lemma~\ref{lemma:numberOfKstars}.
Thus, listing all proper $(t+1)$-stars takes $O(n^{t+3} \m_0(c-1))$ time.

To enumerate all subgraphs admitting a good $(t, 2)$-partition we use the defintion of 
the good $(t, 2)$-partition and the $c$-closure condition.
For two edges $e, f$ there are at most $2c$ vertices that are non-adjacent to 
either $e$ or $f$ by the complementary $c$-closure property.
After fixing a set $A$ of size at most $t$ and particular two edges $e,f$,
by brute-force we can find all subgraphs with treewidth at most $t$ that admit a good $(t, 2)$-partition 
with the set $A$ and the edges $e$ and $f$.
Since there are at most $2c$ vertices over which we have to apply the bruce-force this takes $O(2^{2c})$ time.
In total, going over all sets of size at most $t$ and every two edges $e,f$ takes $O(n^{t+4} 2^{2c})$ time. 

\begin{restatable}{corollary}{thmTreewidth}[Restatement of Theorem \ref{thm:max-treewidth}]
	For $c$-closed graphs and a fixed $t \ge 0$, there is an algorihtm running in time $O(n^{t+4}2^{2c})$ that outputs a set
	containing all maximal induced subgraphs with co-treewidth $\le t$.
\end{restatable}

\smallskip \noindent \textbf{Co-forests} 
Recall that the class of forests is equivalently defined as the class graphs with treewidth at most $1$, or as the class of graph with degeneracy at most $1$. 
In Appendix~\ref{section:forest}, we give stronger bound than the one given in Section~\ref{section:treewidth} and Section~\ref{section:degeneracy} for enumerating maximum co-forests in a $c$-closed graph.

\newpage
\appendix

\section{Moon-Moser Theorem}
\label{sec:proofOfMoonMoser}
\begin{theorem}\label{thm:moonMozer}
	$\m_0(\numOfV) \le 3^{\numOfV/3} \le 1.443^{\numOfV}$.	
\end{theorem}
\begin{proof}
	We prove that $\m_0(\numOfV) \le 3^{\numOfV/3}$ by induction on $\numOfV$.
	Let $G$ be a graph on $\numOfV$ vertices and $v$ a vertex of minimum degree $\ell$ in $G$.
	Any maximal independent set $I$ intersects $N[v]$ in some vertex $w$. 
	As $I$ is a maximal independent set in $G$ then $I\setminus w$ is a maximal independent set in  $G\setminus N[w]$.
	Thus, we get the following recursive bound
	$$\m_0(G) \le \sum_{w\in N[v]} \m_0(G\setminus N[w]) \le \sum_{w\in N[v]} \m_0(\numOfV-|N[w]|) \le (\ell+1) \m_0(\numOfV-(\ell + 1)) \,,$$
	where in the last inequality we use $\m_0(\numOfV-|N[w]|) \le \m_0(\numOfV-(\ell+1))$ for all $w\in N[v]$ since $\ell$ is the minimum degree.
	By induction, we have $(\ell+1) \m_0(\numOfV-(\ell + 1)) \le (\ell + 1)\cdot 3^{\frac{\numOfV-(\ell+1)}{3}}$.
	The theorem follows since 
	$3^{\frac{\numOfV}{3}} (\ell+1) 3^{\frac{-(\ell+1)}{3}} \le 3^{\frac{\numOfV}{3}}  $ for all $\ell \in \mathbb{N}$.
\end{proof}

The proofs by Miller and M\" uller~\cite{millerMuller}
  and Moon and Moser~\cite{moonmoser} give a more refined bound by
  distinguishing the case analysis based on the divisibility of $n$ by
  $3$.


\section{Counting maximal subgraphs with degree at most 1}
\label{section:combinatorics}

In this section we prove Theorem~\ref{thm:generalizedInducedMatching} (or equivalently Theorem~\ref{thm:generalizedInducedMatching2}).
As this section does not use $c$-closure, we use $n$ instead of $\numOfV$ for the number of vertices in an arbitrary graph $G$ ($G$ is not necessarily $c$-closed).

\smallskip
We say that a set $S \subseteq V(G)$ is a \emph{generalized induced matching} 
if $\Delta(G[S]) \le 1$.
Moreover $S$ is a maximal generalized induced matching in $G$ if there is no set $S'\subseteq V(G)$
 such that $S\subset S'$ and $S'$ is a generalized induced matching.

We are interested in the number of maximal generalized induced matchings in a graph $G$, i.e., $\m_1(G)$. 
For a generalized induced matching $S$, we say that $v\in S$ is \emph{unmatched} if $v$ has no neighbors in $S$, 
and \emph{matched} if $v$ has a neighbor in $S$ -- such a neighbor is unique. 
A useful way to think about the \emph{maximal} generalized induced matchings is following:
\begin{observation}\label{obs:seesTwo}
	Let $S$ be a maximal generalized induced matching in $G$.
	Then, each vertex $w \in V(G)\setminus S$ is adjacent to either a matched vertex in $S$ 
	or two unmatched vertices of $S$.

	A converse holds as well. 
	Suppose that $S$ is a generalized induced matching.
	If every vertex $w\not \in S$ is adjacent 
	to a matched vertex in $S$ or at least two unmatched vertices, then $S$ is maximal. 
\end{observation} 

Before we prove the main theorem, we prove three simple lemmas and an easy proposition.

\begin{lemma}\label{lemma:disconnected}
	Let $U$ be a connected component of a graph $G$. 
	Then, $\m_1(G) = \m_1(G[U]) \cdot \m_1(G\setminus U)$.
\end{lemma}
\begin{proof}
	Any maximal generalized induced matching $S$ in $G$ 
	is the disjoint union of a maximal generalized induced matching $S\cap U$ in $G[U]$,
	and a maximal generalized induced matching $S\setminus U$ in $G\setminus U$. 
\end{proof}

\begin{lemma}\label{lemma:twin}
	Let $u, v$ be twin vertices in $G$, i.e., $uv\in E$ and $N(v) = N(u)$.
	Then, $\m_1(G) \le \m_1(G\setminus uv)$.
\end{lemma}
The lemma states that disconnecting twin vertices in a graph cannot decrease the number of maximal generalized induced matchings. 
\begin{proof}
	Let $S$ be a maximal generalized induced matching in $G$. 
	It suffices to prove that $S$ induces a maximal generalized matching in $G \setminus uv$. 
	If $S$ does not contain $u$ nor $v$, then $S$ is a maximal generalized induced matching in $G\setminus uv$. 
	Without loss of generality, $u \in S$. 
	Note that $u$ is matched: if not, then $S \cup v$ is a generalized induced matching since $u$ and $v$ are twins.

	\textit{Case 1:} The neighbor of $u$ in $S$ is $w$, $w\neq v$.
	Then, $S$ is a generalized induced matching in $G\setminus uv$ with the
	same number of edges. 
	$S$ is still maximal since $u$ and $v$ are twins.

	\textit{Case 2:} The neighbor of $u$ in $S$ is $v$. 
	Then, $S$ is an induced matching in $G\setminus uv$
	with one less edge than the generalized induced matching $S$ in $G$. 
	By Observation~\ref{obs:seesTwo} and since $u$ and $v$ are twins, $S$ is maximal.
\end{proof}

\begin{lemma}\label{lemma:domination}
		Let $S$ be a maximal generalized induced matching in $G$ and let $v\in S$.
		Then, either $|N(u) \setminus  N[v]| > 0$ for all $u \in N(v)$ or $v$ is matched in $S$. 
\end{lemma}
In a graph $G$, we say that a vertex $v$ \emph{dominates}  a vertex $u$ if $N(u) \subseteq N(v)$.
The lemma states that if $v$ dominates a vertex in its neighborhood, then $v$ is always matched in a maximal generalized induced matching.
\begin{proof}
	For the sake of contradiction suppose that $v$ is unmatched and 
	that for some $w \in N(v)$ it holds $N(w) \subseteq N[v]$.
	As $v$ is unmatched it holds that $u \not \in S$, for all $u \in N(v)$.
	Since $N(w) \subseteq N[v]$ it follows $S\cup w$  is also a generalized induced matching.
	A contradiction with maximality of $S$.  
\end{proof}

\begin{proposition}\label{prop:baseCase}
	Let $G$ be a graph and suppose that $n = |V(G)| \le 5$. 
	Then, $\m_1(G) \le 10^{n/5}$.
\end{proposition}
For missing definitions in the following proof we refer to~\cite{davey2002introduction}.
\begin{proof}
	The proposition is trivial to check for $n\le 3$.
	Given a ground set $A$ denote with ${\cal P} (A)$ the family of all subsets of $A$.
	${\cal P} (A)$ admits a natural partial ordering by the inclusion. 

	We observe that the set of all maximal generalized induced matchings is an antichain in ${\cal P}(V(G))$ (or any other type of maximal sets).
	If $n = 4$ (resp. $n=5$), then the maximum size of an antichain in ${\cal P}(V(G))$ is $\binom{4}{2} = 6 < 10^{4/5}$ (resp. $\binom{5}{2} = 10 = 10^{5/5}$). 
\end{proof}

\begin{theorem}\label{thm:generalizedInducedMatching2}
  $\m_1(n) \le 10^{n/5} \le 1.585^{n}$.
\end{theorem}
\begin{proof}
	We prove the result by induction on the number of vertices.
	Let $G$ be a graph on $n$ vertices. 
	Proposition~\ref{prop:baseCase} is the base case of the induction and allows us to assume that $|V(G)| \ge 6$.
	By Lemma~\ref{lemma:disconnected} we assume that $G$ is connected.
	Observation~\ref{obs:seesTwo} is used throughout the proof implicitly.
	We will consider several different cases based on the degree of vertices in $G$. 

	\textbf{Case A: there exists a vertex $v\in  V(G)$ with $\deg(v) =1$}.
	Denote with $w$ the unique neighbor of $v$ in $G$. 
	Since $S$ is maximal, it contains at least one of $v, w$. 
	Moreover, if $v\not \in S$, then $w$ is matched in $S$ by Lemma~\ref{lemma:domination}. 
	Thus, either $w \in S$ and $w$ is matched in $S$ or $w\not\in S$ and $v \in S$.

	If $w\in S$ and $w$ is matched then there is $u\in N(w)$ such that $u \in S$.
	In this case, $S\setminus \{w,u\}$ is a maximal generalized induced matching in $G\setminus N[w,u]$.
	If $v\in S$ and $w\not \in S$,
	then $S\setminus v$ is a maximal generalized induced matching in $G\setminus \{v, w\} = G\setminus N[v]$.
	Combining the two, we obtain the following recursive upper bound on $\m_1(G)$:
	\begin{equation*}
		\m_1(G) \le \m_1(G\setminus \{v,w\}) + \sum_{u\in N(w)} \m_1(G\setminus N[w, u]) \,.
	\end{equation*}
	As $|N[w,u]| \ge |N[w]| \ge \deg(w)+1$ we have $\m_1(G\setminus N[w,u]) \le \m_1(n-\deg(w)-1)$ and 
	$$
		\m_1(G) \le \m_1(n-2) + \deg(w) \m_1(n-\deg(w)-1) \,.
	$$
	By induction we have 
	\begin{equation*}
		\m_1(G) \le 10^{n/5} \left( 10^{-2/5} + \deg(w) \cdot 10^{-(\deg(w)+1)/5}\right)\,.
	\end{equation*}
	Since $10^{-2/5} + x \cdot 10^{-(x+1)/5} < 1$, for all $x \ge 1$, this case is proved. 
	Note that we proved a stronger statement:
	if $\deg(v) =1$ for $v\in V(G)$, then the number of maximal generalized induced matchings is 
	$10^{n/5} \left( 10^{-2/5} + \deg(w) \cdot 10^{-(\deg(w)+1)/5}\right) < 10^{n/5} \cdot \left( 10^{-2/5} + 2\cdot 10^{-3/5}\right)$ 
	where $w$ is the neighbor of $v$ in $G$.
	We will use the stronger statement in one of the remaining cases. 

	\medskip
	\smallskip \noindent \textbf{Recursive bound}
	We give a generic recursive bound for $\m_1(G)$ that will be useful for several cases. 
	Let $v$  be an arbitrary vertex.
	For a maximal generalized induced matching $S$ we have the following possibilities. 
	\begin{itemize}
		\item $S$ does not contain $v$. 
			Then, $S$ is also a maximal generalized induced matching in $G\setminus v$. 
			Hence, the number of maximal generalized induced matchings $S$ in $G$ that do not contain $v$ is at most $\m_1(G\setminus v)$.
		\item $S$ contains $v$ and $v$ is unmatched in $S$. 
			Then, $S\setminus v$ is a maximal generalized induced matching in $G\setminus N[v]$.
			The number of such sets $S$ in $G$ is at most $\m_1(G\setminus N[v])$.
		\item $S$ contains $v$ and $v$ is matched to $w$ in $S$. 
			Then, $S\setminus \{v, w\}$ is a maximal generalized induced matching in $G\setminus N[v,w]$.
			The number of such sets $S$ in $G$ is at most $\sum_{w\in N(v)} \m_1(G\setminus N[v,w])$.
	\end{itemize}
	We obtain the following bound on $\m_1(G)$:
	$$
	\begin{aligned}
		\m_1(G) &\le \m_1(G\setminus v) + \m_1(G\setminus N[v]) + \sum_{w \in N(v)}\m_1(G\setminus N[v,w])
	\end{aligned}
	$$
	By Lemma~\ref{lemma:domination}, if there is a vertex $w\in N(v)$ such that $N(w) \subseteq N[v]$,
	then we cannot have $v\in S$ and $v$ unmatched. 
	Therefore, in this case the stronger bound applies:
	$$
	\begin{aligned}
		\m_1(G) &\le \m_1(G\setminus v) + \sum_{w \in N(v)}\m_1(G\setminus N[v,w])
	\end{aligned}
	$$

	\textbf{Case B: $\Delta(G) \ge 6$.}
	Let $v$ be a vertex of degree at least $6$.
	Since $|N[v,w]| \ge |N[v]| = \deg(v) + 1$, 
	it follows that $\m_1(G\setminus N[v,w])\le \m_1(G\setminus N[v]) \le \m_1(n-\deg(v)-1)$.
	Using the previous in the (weaker) recursive bound gives
	$$
	\begin{aligned}
		\m_1(G) &\le \m_1(n-1) + (\deg(v)+1) \cdot \m_1(n-\deg(v) -1)\,. 
	\end{aligned}
	$$
	By induction and since $10^{-1/5} + (x + 1) \cdot 10^{-(x+1)/5}\le 1$ for $x\ge 6$ we have
	$$
	\m_1(G) \le 10^{n/5} \left( 10^{-1/5} + (\deg(v) + 1) \cdot 10^{-(\deg(v)+1)/5}\right) < 10^{n/5}\,.
	$$

	\textbf{Case C: $\Delta(G) = 5$.} 
	Let $v$ be a vertex of degree $5$. 
	Since $G$ is connected, either $|V(G)| = 6$ or there is a vertex $w \in N(v)$ such that $|N(w) \setminus N(v)| \ge 1$. 
	
	If $|V(G)| = 6$, then a maximal generalized induced matching is either an edge $vw$ for some $w\in N(v)$ or 
	a maximal generalized induced matching in $G\setminus v$. 
	Hence, by Proposition~\ref{prop:baseCase} we have $\m_1(G) \le 5 + 10^{5/5} < 10^{6/5}$.
	
	For the rest of this case we assume that $|V(G)| > 6$.
	Consequently, there is a vertex $w \in N(v)$ such that $|N(w) \setminus N(v)| \ge 1$.
	For the four vertices $u \in N(v)\setminus w$ we use the same bound as before $|N[u,v]| \ge |N[v]| = \deg(v) + 1 = 6$.
	Since $|N(w)\setminus N(v)| \ge 1$, we have a stronger bound $|N[v,w]|\ge |N[v]|+1 = \deg(v) + 2$. 
	Thus, $\m_1(G\setminus N[v,w]) \le \m_1(n-\deg(v) -2 )$.
	By the (weaker) recursive bound and induction we have
	$$
	\begin{aligned}
		\m_1(G) &\le \m_1(n-1) + \m_1 (n-6) + 4 \cdot  \m_1(n-6) + \m_1(n-7) \\
		&\le 10^{n/5} \left( 10^{-1/5} + 5 \cdot 10^{-6/5} + 10^{-7/5} \right) < 10^{n/5} \,.
	\end{aligned}
	$$

	\textbf{Case D: $\Delta(G) = 4$.}
	Let $v$ be a vertex of degree $4$. 
	We consider two subcases. 
	In the first case we assume that each $w\in N(v)$ has a neighbor outside $N[v]$.
	Otherwise, for some $w\in N(v)$ it holds $N[w] \subseteq N[v]$ -- the second case. 

	\textit{Case D.1: For all $w \in N(v)$ it holds $|N(w) \setminus  N[v]|\ge 1$.}
	Therefore, $|N[v,w]|\ge |N[v]| + 1 = \deg(v) + 2 = 6$. 
	By the (weaker) recursive bound and induction we have
	$$
	\begin{aligned}
		\m_1(G) &\le \m_1(n-1) + \m_1(n-5) + 4\cdot \m_1(n-6)\\
		& \le 10^{n/5} \left(10^{-1/5} + 10^{-5/5} + 4\cdot 10^{-6/5} \right) < 10^{n/5}\,.
	\end{aligned}
	$$

	\textit{Case D.2: For some $w\in N(v)$ we have $N(w)\subseteq N[v]$.}
	Since $|V(G)| \ge 6$ and since $G$ is connected, 
	for some $u \in N(v)$ we have $|N(u)\setminus N[v]|\ge 1$ and thus $|N[v,u]| \ge |N[v]| + 1 = \deg(v) + 2 = 6$. 
	Combining it with the (stronger) recursive bound, and by induction gives
	$$
	\begin{aligned}
		\m_1(G) &\le \m_1(n-1) + 3\cdot \m_1(n-5) + \m_1(n-6) \\
		&\le 10^{n/5} \left(10^{-1/5} + 3\cdot 10^{-5/5} + 10^{-6/5} \right) < 10^{n/5} \,.
	\end{aligned}
	$$

	\textbf{Case E: $\Delta(G) = 3$.} 
	Let $v$ be a vertex of degree $3$ and denote with $w_1, w_2, w_3$ its neighbors.
	Since $G$ is connected and $|V(G)|\ge 6$ at least one $w_i$ has a neighbor outside of $N[v]$.
	Moreover, by case \textbf{A} there are no vertices of degree $1$ in $G$.

	We consider five subcases. 
	In the first three the cases, at least one of $w_1, w_2, w_3$ has degree $2$ in $G$.
	In the last two, $\deg(w_i) = 3$ for every $i \in[3]$.

	\textit{Case E.1: $\deg(w_3) = 2$, $N(w_3) \subseteq N[v]$ and for $i\in[2]$ it holds $|N(w_i) \setminus N[v]| \ge 1$ .}
	For $i\in [2]$ we have $|N[v, w_i]| \ge |N[v]| + 1 \ge \deg(v) + 2  = 5$. 
	Moreover, $|N[v, w_3]| = \deg(v) + 1 = 4$.
	By the (stronger) recursive bound, and induction we have 
	$$
	\begin{aligned}
		\m_1(G) &\le \m_1(n-1) + 2\cdot \m_1(n-5) + \m_1(n-4) \\
		&\le  10^{n/5} \left(10^{-1/5} +  2\cdot 10^{-5/5}  + 10^{-4/5}\right) < 10^{n/5}  \,.
	\end{aligned}
	$$

	\textit{Case E.2: $\deg(w_3) = 2$, for $i\in \{2, 3\}$  it holds $N(w_i) \subseteq N[v]$, and $|N(w_1) \setminus N[v]|\ge 1$.}
	Since $\Delta(G) =3$ vertex $w_1$ can be adjacent to at most one of $w_2, w_3$. 
	If $w_3$ is adjacent to $w_1$, then $\deg(w_2) = 1$.
	Hence, $w_3$ in non-adjacent to $w_1$ and $w_3$ is adjacent to $w_2$.
	By Lemma~\ref{lemma:twin} we assume that $w_2$ and $w_3$ are not twins. 
	Hence, $w_2$ is adjacent to $w_1$. 

	We use a refined version of the strong recursive bound.
	In particular, we refine the term corresponding to the case where $v\not \in S$.
	If $v\not \in S$, then by maximality at least one of the following cases holds:
	\begin{itemize}
		\item $w_1, w_3 \in S$ and $w_1$ is unmatched,
		\item $w_1, w_3 \in S$ and $w_1$ is matched to its neighbor $t$ with $t\in S\setminus N[v]$, 
		\item $w_1, w_2 \in S$,
		\item $w_2, w_3 \in S$. 
	\end{itemize}
	We obtain the bound
	$$
	\begin{aligned}
		\m_1(G) &\le \m_1(G\setminus N[w_1, w_3]) + \sum_{t\in N(w_1)\setminus N[v]} \m_1(G\setminus N[w_1, t, w_3]) + \m_1(G\setminus N[w_1, w_2])
		\\ &+ \m_1(G\setminus N[w_2, w_3]) + \sum_{i\in [3]}\m_1(G\setminus N[v,w_i])\,.
	\end{aligned}
	$$
	Note that $|N(w_1)\setminus N[v]| = 1$ as $v,
w_2 \in N(w_1) \cap N[v]$. 
	Since $|N[w_1, t, w_3]| \ge |N[w_1, w_3]| = |N[w_1, w_2]| = |N[w_1, v]| = 5$ we have
	$$
	\begin{aligned}
		\m_1(G) &\le \m_1(n-5) + \m_1(n-5) + \m_1(n-5)
		\\ &+ \m_1(n-4) + 2\cdot \m_1(n-4) + \m_1(n-5)\,.
	\end{aligned}
	$$
	By induction we have
	$
	\displaystyle	\m_1(G) \le 10^{n/5} \left(4\cdot 10^{-5/5} + 3\cdot 10^{-4/5}\right)\,.
	$
	The case is proved since $4\cdot 10^{-1} + 3\cdot 10^{-4/5} < 1$.

	\textit{Case E.3: $\deg(w_3) = 2$ and for all $i\in [3]$ it holds $|N(w_i) \setminus N[v]| \ge 1$.}
	Let $u_3$ be the neighbor of $w_3$ different than $v$. 
	Note that $u_3 \not \in N[v]$ by the assumption.
	By case \textbf{A} it holds $\deg(u_3) \in \{2,3\}$.
	We use a refined version of the weaker recursive bound. 
	More precisely, we refine the term $\m_1(G\setminus N[v])$ corresponding to $v\in S$ and $v$ unmatched. 
	If $v\in S$ and $v$ unmatched, then by maximality it follows that either $u_3 \in S$ and $u_3$ unmatched 
	or for  some $t\in N(u_3) \setminus N[v]$ we have $u_3t\in S$.
	The recursion becomes
	$$
	\begin{aligned}
		\m_1(G) &\le \m_1(G\setminus v) \\
		&+ \m_1(G\setminus N[v, u_3]) + \sum_{t\in N(u_3)\setminus N[v]} \m_1(G\setminus N[v, u_3, t]) \\
		&+ \sum_{i\in [3]}\m_1(G\setminus N[v,w_i])\,.
	\end{aligned}
	$$
	Let $x= |N(u_3)\setminus N[v]|$, i.e, $x$ is the number of vertices that are adjacent to $u_3$ but not $v$.
	Since $\deg(u_3) \in \{2,3 \}$ it follows that $x\in \{0,1,2\}$.
	By definition of $x$ and since $\{v, w_1, w_2, w_3, u_3\}\subseteq N[v,u_3]$ it holds that $|N[v,u_3]| \ge x+5$.
	Thus, $| N[v,u_3,t]| \ge x+5$ as well. The induction gives:
	$$
		\m_1(G\setminus N[v, u_3]) + \sum_{t\in N(u_3)\setminus N[v]} \m_1(G\setminus N[v, u_3, t]) 
		\le (x+1) \cdot \m_1(n-x-5) \le (x+1)\cdot 10^{(n-x-5)/5}		\,.
	$$
	By induction and since $|N[v,w_i]| \ge 5$ for all $i\in [3]$ we also have
	$$
		\sum_{i \in [3]}\m_1(G\setminus N[v,w_i]) \le 3\cdot 10^{(n-5)/5}\,.
	$$
	Note that the degree of $w_3$ in graph $G\setminus v$ is $1$. 
	Hence, we can apply the following bound given in case \textbf{A}:
	$$
	\m_1(G \setminus v) \le 10^{(n-1)/5} \left( 10^{-2/5} + 2\cdot 10^{-3/5}\right)
	$$
	Combining the above three we have:
	$$
		\m_1(G) \le 10^{(n-1)/5} \left( 10^{-2/5} + 2 \cdot 10^{-3/5}\right) + (x+1) \cdot 10^{(n-x-5)/5} + 3\cdot 10^{(n-5)/5}\,.
	$$
	For all three possible values $\{0,1,2\}$ for $x$ the last is bounded by $10^{n/5}$.
	Therefore, case \textit{E.3} is proved.

	\medskip
	Consider the previous three subcases. 
	The vertex $v$ is an arbitrary vertex of a connected graph $G$.
	In other words, one of the three subcases can be applied  as soon as there is a vertex in $G$ of degree $3$, 
	with a neighbor of degree $2$.
	Therefore, by case \textbf{A} and since $G$ is connected, we may assume for the rest of the proof that $G$ is a cubic graph, 
	i.e., the degree of every vertex in $G$ is $3$.

	Let $v$ be a vertex in the cubic graph $G$. 
	By Lemma~\ref{lemma:twin} for every $w_i \in N[v]$ it holds $|N(w_i)\setminus N[v]| \ge 1$. 
	We will consider the following two possibilities: $G[N(v)]$ is an independent set or $G[N(v)]$ contains exactly one edge.

	\textit{Case E.4: $G[N(v)]$ is an independent set.}
	Equivalently, for all $i\in [3]$ it holds $|N(w_i) \setminus N[v]| = 2$.
	Since $|N[v, w_i]| = 6$ for all $i\in [3]$, we have by the (weak) recursive bound:  
	$$
 	\m_1(G) \le \m_1(n-1) + \m_1(n-4) + 3\cdot \m_1(n-6)\,.
	$$
	Then, by induction
	$
 	\displaystyle	\m_1(G) \le 10^{n/5} \left( 10^{-1/5} + 10^{-4/5} + 3\cdot 10^{-6/5} \right) < 10^{n/5}\,.
	$

	\textit{Case E.5: $G[N(v)]$ contains exactly one edge.}
	Without loss of generality assume that $w_2w_3 \in E(G)$.
	Let $u_2$ be the neighbor of $w_2$ outside of $N[v]$, and analogously define $u_3$. 
	By Lemma~\ref{lemma:twin} we assume that $u_2 \neq u_3$. 
	Denote with $a,b$ the neighbors of $w_1$. (It is possible that $\{a, b\} = \{u_2, u_3\}$.)

	We again use a refined version of the weak recursive bound. 
	We refine the term $\m_1(G\setminus v)$ corresponding to the case when $v\not \in S$. 
	By maximality, at least one of the following holds:
	\begin{itemize}
		\item for one of $i\in \{2,3\}$ we have $w_1, w_i \in S$ and both $w_1$ and  $w_i$ are unmatched. 
		Then $S\setminus \{w_1, w_i\}$ is a maximal generalized induced matching in $G\setminus N[w_1, w_i]$.
		\item $w_2, w_3 \in S$. Then $S\setminus \{w_2, w_3\}$ is a maximal generalized induced matching in $G\setminus N[w_2, w_3]$.
		\item For some $i \in [3]$ and some $t\in N(w_i)\setminus N[v]$ it holds $w_i, t\in S$. 
		Then $S\setminus \{w_i, t\}$ is a maximal generalized induced matching in $G\setminus N[w_i, t]$.
	\end{itemize}

		From the above
		$$
		\begin{aligned}
				\m_1(G) \le &+ \sum_{i \in \{2, 3\}} \m_1(G\setminus N[w_1, w_i]) \\
				&+ \m_1(G\setminus N[w_2, w_3])\\
				&+ \sum_{t\in \{a,b\}} \m_1(G\setminus N[w_1, t]) + \m_1(G\setminus N[w_2, u_2]) + \m_1(G\setminus N[w_3, u_3])\\
				&+ \m_1(G\setminus N[v]) + \sum_{i \in [3]} \m_1(G\setminus N[v, w_i]) \,.
		\end{aligned}
		$$
		Since $G$ is cubic and by the adjacencies in $G$ the following equalities and inequalities hold: 
		\begin{itemize}
			\item $|N[w_1, w_2]|\ge 6$ and $|N[w_1, w_3]|\ge 6$; 
			\item $|N[w_2, w_3]| = 5$;
			\item $|N[w_1, a]| \ge 5$ and $|N[w_1, b]|\ge 5$; $N[w_2, u_2] \ge 6$; $|N[w_3, u_3]|\ge 6$;
			\item $|N[v]| = 4$; $|N[v, w_1]| = 6$, $|N[v, w_2]| = 5$ and $|N[v, w_3]| = 5$. 
		\end{itemize}
		Hence
		$$
			\m_1(G) \le \m_1(n-4) + 5\cdot \m_1(n-5) + 5\cdot \m_1(n-6)\,.
		$$
		By induction we get
		$
		\displaystyle	\m_1(G) \le 10^{n/5} \left( 10^{-4/5} + 5\cdot 10^{-5/5} + 5\cdot 10^{-6/5}\right) < 10^{n/5}\,.
		$

	\textbf{Case F: $\Delta(G) = 2$.}
	Since $G$ is connected and by case \textbf{A} it follows that $G$ is a cycle.
	Let $v_2\in V(G)$ and denote with $v_1$ and $v_3$ its neighbors.
	We use a refined recursive bound 
	where we refine the term $\m_1(G\setminus v_2)$ corresponding to the maximal generalized induced matchings that do not contain $v_2$.
	If $v_2 \not \in S$ then at least one of the following holds
	\begin{itemize}
	  	\item $v_1, v_3 \in S$ and both $v_1$ and $v_3$ are unmatched. 
	  	\item $v_1 \in S$ and $v_1$ is matched to its neighbor $v_0$, where $v_0 \neq v_2$. 
	  	\item $v_3 \in S$ and $v_3$ is matched to its neighbor $v_4$, where $v_4 \neq v_2$. 
	 \end{itemize}  
	The bound arises
	$$
	\begin{aligned}
		\m_1(G) &\le \m_1(G\setminus N[v_1, v_3]) + \m_1(G\setminus N[v_1, v_0]) + \m_1(G\setminus N[v_3, v_4]) \\
		&+ \m_1(G\setminus N[v_2]) + \m_1(G\setminus N[v_1, v_2]) + \m_1(G\setminus N[v_2, v_3])
	\end{aligned}
	$$ 
	Since $|V(G)|\ge 6$ it follows that $v_0 \neq v_4$. 
	Similarly as before, by induction we obtain
	$$
		\m_1(G) \le 10^{n/5} \left( 10^{-5/5} + 2\cdot 10^{-4/5} + 10^{-3/5} + 2\cdot 10^{-4/5} \right) < 10^{n/5}\,.
	$$
	This completes the proof.	
\end{proof}

\section{Bounded co-degeneracy}
\label{section:degeneracyAppendix}

In this section, we give FPT algorithms for enumerating maximal subgraph with bounded co-degeneracy in a $c$-closed graph. 
As before, we work in the complement of  a $c$-closed graph and look for the maximal subgraphs of bounded degeneracy.
The proof of the FPT bound uses the same lemmas as in the case of bounded treewidth. 
Namely, it is easy to show that  $d$-degenerate graph is either a $4d$-star or admits a good $(4d, 2d)$-partition.
The FPT bound then follows by Lemmas~\ref{lemma:numberOfKstars} and~\ref{lemma:kEdges}.

In comparison with bounded treewidth, we are able to make exponential savings in running time with respect to $c$
since there is a non-trivial combinatorial bound, and there is a polynomial delay algorithm for listing maximal $d$-degenerate subgraphs.

\smallskip \noindent \textbf{Combinatorial bound}
Recall that the maximum the number of maximal $d$-degenerate subgraph with in an arbitrary $\numOfV$-vertex graph is denoted by 
$\dd_d(\numOfV)$.
Pilipczuk and Pilipczuk~\cite{pilipczuk2012finding} show that for every $d$ there is a constant $\gamma_d<2$ such that $\dd_d(\numOfV) \le \gamma_d^{\numOfV}$.
Forests are exactly $1$-degenerate graphs, so we have $\dd_1(\numOfV) = \f(\numOfV) \le 1.8638^{\numOfV}$.

\smallskip \noindent \textbf{FPT bound}
To give the algorithm, we use the same two lemmas
 as in the case of subgraphs of bounded treewidth in the complement of a $c$-closed graph. 
In the case of bounded degeneracy, the dichotomy theorem is easier to prove but it comes at the expense of worse upper bounds.
To make the saving in the base of the exponent we give a stronger version of Lemma~\ref{lemma:kEdges}.

\begin{lemma}\label{lemma:kEdges+}
	Let $G$ be the complement of a $c$-closed graph and let $\ell$ and $k$ be fixed integers. 
	The number of maximal subsets $S \subseteq V$ for which graph $G[S]$ is $d$-degenerate and admits a good $(\ell, k)$-partition,  
	is at most $O(n^{\ell+2k}\cdot \dd_d(kc))$.
\end{lemma}
\begin{proof}
	Let $H$ be a maximal $d$-degenerate subgraph of $G$ that let 
	$e_1, \dots, e_k$  and $A_0, \dots, A_k$ be the edges and sets defining a good $(\ell, k)$-partition of $H$.
	To prove the lemma, it suffices to show that the number of induced subgraphs that admit a good $(\ell, k)$-partition 
	with the same edges $e_1, \dots, e_k$ and the same set $A_0$ is bounded by $\dd_d(\ell + k(c+1))$.
	Namely, $O(n^{\ell+2k}\cdot \dd_d(\ell + kc+k )) 
	\le O(n^{\ell+2k}\cdot 2^{\ell + k} \cdot \dd_d(kc)) 
	= O(n^{\ell+2k}\cdot\dd_d(kc))$.

	Denote with $U$ the vertices of $G$ that are neither incident to the edges $e_1, \dots, e_k$ nor in the set $A_0$, 
	i.e., $U = V(G)\setminus (A_0 \cup_{i=1}^k e_i)$. 
	By definition of a good $(\ell, k)$-partition, for any induced subgraph with a good $(\ell, k)$-partition 
	$e_1, \dots, e_k$ and $A_0, A'_1, \dots A'_k$ it holds $A'_i\subseteq U \setminus N_G(e_i)$ for each $i \in [k]$.
	Since $G$ is complement of a $c$-closed graph, it follows that $|U\setminus N(e_i)| \le c-1$ for each $i\in[k]$. 
	Hence, $H$ is also a maximal subgraph in graph induced by $A_0 \cup \left(\bigcup_{i = 1}^k (e_i \cup (U\setminus N[e_i])) \right)$.
	As $\left|A_0 \cup \left(\bigcup_{i = 1}^k (e_i \cup (U\setminus N[e_i]))\right)\right| \le \ell + k(c+1)$ 
	we conclude that there are at most $\dd_d(\ell + k(c+1))$ subgraph $H$ with desired properties.
	The lemma follows.
\end{proof}

To prove that every $d$-degenerate graph is either a $4d$-star or admits a good $(4d, 2d)$-partition we need an easy proposition.

\begin{proposition}\label{prop:uniformDegree-edges}
	If $H$ is a graph of degeneracy at most $d$ then every induced subgraph $H'$ it holds $|E(H')| \le d|V(H')|$.
\end{proposition}

\begin{lemma}\label{lemma:uniformDegreePartition}
	Let $H$ be a graph of degeneracy at most $d$. 
	Then $H$ is either a $4d$-star or $H$ admits a good $(4d, 2d)$-partition. 
\end{lemma}
\begin{proof}
	Let $M$ be the maximum size matching in $H$.
	If $|M| \le 2d$ then the vertices incident with the edges in $M$ form a vertex cover of size at most $4d$.
	In this case, $H$ is trivially a $4d$-star.
	So assume that $|M|\ge 2d$ and consider  $2d$ arbitrary edges from $M$, say $e_1, \dots, e_{2d}$.

	In order to prove the lemma it suffices to show that $|\cap_{i=1}^{2d} N(e_i) | \le 4d$:
	namely, we set $A_0 = \bigcap_{i=1}^{2d} N(e_i)$; then each $v \in V(H)\setminus (A_0 \bigcup_{i=1}^{2d}e_i)$ 
	is non-adjacent to at least one edge $e_i$ and we can assign $v$ to $A_i$.

	Denote with $\ell = |\cap_{i=1}^{2d} N(e_i)|$. Our goal is to show that $\ell \le 4d$.
	Let $Z =\cup_{i=1}^{2d} e_i$ be the set of vertices incident to the edges $e_1, \dots, e_{2d}$. 
	It holds $|Z| = 4d$.
	As each vertex in $\cap_{i=1}^{2d} N(e_i)$ is adjacent to every edge $e_i$, 
	the number of edges in $G[A_0 \cup Z]$ is at least $\ell \cdot 2d$.
	Since $H$ has degeneracy at most $d$, by Proposition~\ref{prop:uniformDegree-edges} it holds 
	$$
		2\ell d \le |E(G[A_0 \cup Z])| \le d|A_0 \cup Z| = 
		d(\ell+4d) = d\ell + 4d^2 \,.
	$$
	Hence, $d\ell \le 4d^2$ and $\ell \le 4d$.
\end{proof}

We are ready to prove the theorem.
\FPTboundeDegeneracy*
\begin{proof}
By Lemma~\ref{lemma:uniformDegreePartition} any maximal induced subgraph with degeneracy at most $d$ is either a $4d$-star (and hence a proper $4d+1$-star)
or admits a good $(4d, 2d)$-partition.
By Lemma~\ref{lemma:numberOfKstars} the number of proper $4d+1$-starts with maximal set of tails is $2n^{4d+1}\m_0(c-1)$.
By Lemma~\ref{lemma:kEdges+} there are at most $O(n^{8d} \dd_d(2dc))$ subgraphs that admit a good $(4d, 2d)$-partition in $G$. 
\end{proof}
Since the minimum degeneracy of $K_{a,b}$ is $\min \{a,b\}$, Example~\ref{exampleTW} shows that the dependency $n^t$ is necessary.

\smallskip \noindent \textbf{Enumeration}
Maximal $d$-degenerate subgraphs can be listed in time $O(mn^{d+2})$ per  maximal subgraph~\cite{conte2019proximity}.
We obtain the following corollary.
\enumerationDegeneracy*

\section{Bounded local co-treewidth}
\label{section:localTreeWidth}
We use the lemmas and ideas present above for the subgraphs of bounded treewidth 
to show that the similar results hold for the subgraphs of bounded \emph{local} treewidth.
First, we give a corollary of Lemma~\ref{lemma:kEdges} and then we recall the definition of locally bounded treewidth.

\begin{corollary}\label{cor:diam6}
	Let $G$ be a complement of a $c$-closed graph. 
	Then, the number of subsets $S$, for which either 
		\begin{itemize}
			\item $G[S]$ contains at least two non-trivial connected components, or
			\item the diameter of some connected component in $G[S]$ is at least $6$
		\end{itemize}
	 is bounded by $n^4\cdot 2^{2(c-1)}$.
\end{corollary}
\begin{proof}[Proof of Corollary~\ref{cor:diam6}]
We show that in both cases $G[S]$ admits a good $(0,2)$-partition. The corollary then follows by Lemma~\ref{lemma:kEdges}.
If $G[S]$ contains two non-trivial connected components then $G[S]$ clearly admits a good $(0, 2)$-partition. 

Suppose that $G[S]$ contains two vertices $v,u$ in the same component that are at distance at least $6$. 
Since $u,v$ are in the same connected component there are different vertices $u', v'\in S\setminus \{u,v\}$ such that $uu', vv' \in E(G)$ (say the neighbors of $u, v$ on the shortest $u - v$ path).
As $u$ and $v$ are at distance at least $6$ it follows that $N(uu') \cap N(vv') \cap S = \emptyset$. 
In other words, any vertex is either non-adjacent to $uu'$ or $vv'$. 
Thus, $G[S]$ admits a good $(0, 2)$-partition with edges $e_1 = uu'$ and $e_2 = vv'$.
\end{proof}

Informally, the corollary states that if we are are counting (finding) sparse subgraphs in the complement of a $c$-closed graph,
we only need to worry about the subgraphs with a small diameter.

\smallskip \noindent \textbf{Local treewidth}
The local treewidth of a graph $G=(V,E)$ is the function $\ltw^G: \N \to \N$ that
associates with every $r\in \N$ the maximal treewidth of an $r$-neighborhood
in $G$, see~\cite{grohe2003local,nevsetvril2008structural}.
More formally, the $r$-neighborhood $N_r(v)$ of a vertex $v\in V$
is the set of all vertices $u \in V$ at distance at most $r$ from $v$.
Then
$$
\ltw^G(r) := \max \{ \tw(G[N_r(v)]) : v\in V\}\,.
$$
We say that a class of graph $\CC$ has \emph{bounded} local treewidth, 
if  there is a function $f: \N \to \N$ such that for all $G\in \CC$ and $r\in \N$ it holds $\ltw^G(r) \le f(r)$.
Suppose that $\CC$ is a class of graphs with locally bounded treewidth for a function $f$ with $f(1) = f(2) = f(3) = f(4) = f(5) = t$. 
(Equality is needed to ensure that the class $\CC$ contains all proper $t+1$-stars. 
By assuming other conditions, we can relax this assumption.)
We obtain the following theorem.

\begin{theorem}\label{thm:localTW}
	Let $\CC$ be a class of graphs of bounded local treewidth as defined above.  
	Let $G$ be a complement of a $c$-closed graph. 
	Then there are at most $O(n^{t+4}4^{c-1})$ maximal induced subgraphs of $G$ that are in $\CC$. 
\end{theorem}
\begin{proof}
	Let $S$ be a subset of $V(G)$ such that $G[S] \in \CC$.
	By Corollary~\ref{cor:diam6}, there are at most $n^4\cdot 2^{2c-2}$ subsets $S$ for which $G[S]$ has diameter more than $5$.
	On the other hand, if $\diam(G[S])\le 5$, then $\tw(G[S]) \le t$ since $\CC$ has locally bounded treewidth.
	In this case, we can prove the theorem exactly the same as Theorem~\ref{thm:boundedTW} by using Claim~\ref{claim:separatorOrStar}.
\end{proof}

\section{Co-forests}
\label{section:forest}

We consider the dense subgraphs $G[S]$ with at most $|S|-1$ non-edges. 
As we have seen in Example~\ref{example1} if we do not require any structural assumption on the non-edges,
then we cannot hope to enumerate such dense subgraphs in FPT time with respect to $c$.
On the contrary, if we require that the non-edges form a forest then we can get a positive result.
Intuitively, the difference is that in the latter case the non-edges are uniformly distributed within the dense subgraph 
while in the former case the non-edges can be concentrated in a small but not-so-dense part of the subgraph.

\smallskip
For brevity, we are working with maximal forests in the complement of $c$-closed graphs.
We use existing results for the combinatorial bound and enumeration.
An FPT bound follows by separately counting stars and the forests that are not stars. 
We show that counting stars reduces to counting independent sets.
If a forest is not a star then it either contains a path on four vertices, or two non-trivial components.
We denote a path on $4$ vertices by $P_4$.
To count the forests containing a $P_4$ or two non-trivial components we use the following observation. 
Any such forest contains two edges $e, f$ with the property that any other vertex 
is non-adjacent to either the endpoints of $e$ or the endpoints of $f$.
We use the complementary $c$-closure to observe that any forest admitting two such two edges, 
is contained in a set of at most $2c-2$ vertices -- the non-neighbors of $e$ and $f$.

\smallskip \noindent \textbf{Combinatorial bound}
As counting maximal stars reduces to counting maximal independent set, in this case we use $\m_0(\numOfV)$ and the Moon-Moser theorem.
For counting maximal forests different than stars, we use $\f(\numOfV)$ -- the maximum number of maximal induced forests in a graph on $\numOfV$ vertices.
Currently, the best bound is $\f(\numOfV) \le 1.8638^{\numOfV}$~\cite{fomin2008minimum}.
It is known that $\m_0(\numOfV) \le 3^{\numOfV/3} < 105^{\numOfV/10} \le \f(\numOfV)$.

\smallskip \noindent \textbf{FPT bound}
We start by counting the maximal stars in the complement of a $c$-closed graph. 
A non-standard definition of a star is used: a star is a graph that can be obtained as a (vertex disjoint) union of a tree with diameter at most $2$ 
and an independent set.

\begin{lemma}[Stars]
\label{lemma:maximumNumOfStars}
	Let $G$ be a complement of a $c$-closed graph. 
	The number of maximal induced stars in $G$ is bounded by $n^3\cdot \m_0(c-1)$.
\end{lemma}
\begin{proof}
	For a vertex $v\in V(G)$, we show that the number of maximal stars for which $v$ is a center is bounded by $n^2 \cdot \m_0(c-1)$.
	A \emph{center} of a star is any vertex with maximum degree (center is unique whenever there is a vertex with degree at least $2$). 

	Let $S$ be a set inducing a star such that $v$ is a center of $G[S]$.
	By our definition of a star, we have that $S\setminus v$ is an independent set.
	Moreover, $S\setminus v$ is a maximal independent set in graph $G\setminus v$: 
	suppose not and let $u \in V(G)\setminus v$ be a vertex such that 
	$S \cup u \setminus v$ is an independent set in $G\setminus v$, then $S\cup u$ induces a star in $G$ regardless of the adjacency of $u$ and $v$.
	By Theorem~\ref{thm:cliquesIntro}, the number of maximal independent sets in $G\setminus v$ is at most $(n-1)^2 \m_0(c-1)$. 
	The lemma follows.
\end{proof}

We show that the dependence on $n^3$ cannot be improved unless the bound in Theorem~\ref{thm:cliquesIntro} is improved. 
\begin{example}
	Adding an isolated vertex to a $c$-closed graph produces a larger $c$-closed graph. 
	Equivalently, adding a universal vertex (adjacent to all other vertices) to the complement of a $c$-closed graph 
	produces a larger co-$c$-closed graph. 
	
	Let $G$ be the complement of a $c$-closed graph on $\frac{2}{3}n$ vertices. 
	Denote with $G^+$ the graph obtained by adding $\frac{n}{3}$ universal vertices to $G$. 
	The number of maximal induced stars in $G^+$ is at least $\frac{n}{3}$ times larger 
	than the number of maximal independent sets in $G$ as any maximal independent set in $G$ gives rise to $\frac{n}{3}$
	maximal stars in $G^+$.
	Thus, if we start with a graph $G$ having $M$ maximal independent sets 
	we can build graph $G^+$ with $\frac{n}{3}  \cdot M$ maximal induced starts.
\end{example}

We proceed the give an upper bound on the number of forests that contain a $P_4$, 
and the number of forests that contain two non-trivial components.

\begin{lemma}\label{lemma:forestsP4}
	Let $G$ be the complement of a $c$-closed graph. Then the number of maximal induced forests in $G$
	\begin{enumerate}
		\item\label{twoComponents}  with at least two non-trivial components, is at most $n^2 \cdot (c-1)^2 \f(2c-4)$;
		\item\label{P4} containing a $P_4$, is at most $n^3 \cdot (c-1) \f(2c-3)$.
	\end{enumerate}
\end{lemma}
\begin{proof}
\eqref{twoComponents}
	Let $e=ab$ be an edge in $G$.
	We show that the number of maximal forests in $G$ with at least two non-trivial components 
	one of which contains $e$, is at most $(c-1)^2\cdot \f(2c-4)$.
	Any such forest contains 
	an edge $f=cd$ that is in a different connected component than $e$. 
	In particular, vertices $c$ and  $d$ are non-adjacent to $a,b$. 
	Therefore, the edge $f$ is contained in the set $V(G)\setminus N[a,b]$. See Figure~\ref{figure:foretsP4}.
	Since $G$ is the complement of a $c$-closed graph it holds $|V(G)\setminus N[a,b]| < c$.
	Thus, there are at most $(c-1)^2$ possibilities for an edge $f$.
	Fix such an edge $f$.
	To prove~\ref{twoComponents}, 
	it suffices to show that the number of maximal induced forests that contain edges $e$ and $f$ in different components 
	is at most $\f(2c-4)$.

	\begin{figure}[h]
	\centering
	\includegraphics[width = 0.8 \textwidth]{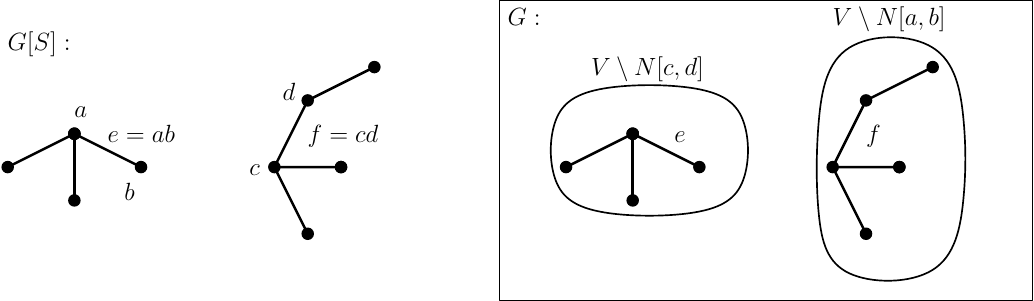}
	\caption{Proof of Lemma~\ref{lemma:forestsP4}. Left: $G[S]$ represents an induced forest with two non-trivial components and two designated edges $e$ and $f$.
	Right: depection of $G[S]$ within $G$. 
	As $G$ is the complement of a $c$-closed graph we have $|V(G)\setminus N[a,b]| \le c-1$ and $|V(G) \setminus N[c,d]| \le c-1$.}
	\label{figure:foretsP4}
	\end{figure}

	Let $S$ be a set inducing a maximal forest with edges $e$ and $f$ in different components. 
	Denote with $X$ the connected component containing $e$.
	Since $X$ is not in the same component as $f$ it follows that $X\subseteq V(G)\setminus N[c,d]$.
	Similarly, we have that $S\setminus X \subseteq V(G) \setminus N[a,b]$.
	Hence, $S$ is also a maximal induced forest in the graph induced by $(V(G)\setminus N[c,d]) \cup (V(G) \setminus N[a,b])$.
	As $G$ is the complement of a $c$-closed graph it follows that $|(V(G)\setminus N[c,d]) \cup (V(G) \setminus N[a,b])| \le c-1 + c-1$.
	At this point we could conclude that the number of such maximal sets $S$ is bounded by $\f(2c-2)$, 
	but we can do a bit better since we are only counting maximal induced forests that contain $e$ and $f$.

	Consider graph $H = G[ (V(G)\setminus N[c,d]) \cup (V(G) \setminus N[a,b]) ]$. 
	Any vertex that is adjacent to both $a,b$ or both $c,d$ cannot be in an induced forest containing $e$ and $f$, 
	so we assume that there are no such vertices in $H$.
	Contract the edges $e$ and $f$ in $H$ to obtain $H'$ and denote with $v_e$ (resp. $v_f$) the vertex obtained by contracting $e$ (resp. $f$).
	Then, for any set $S$ that induces a maximal forest containing $e$ and $f$ in $H$ we have that 
	$S\cup v_e \cup v_f \setminus \{a,b,c,d\}$ induces a maximal forest containing $v_e$ and $v_f$ in $H'$.
	Since $|V(H')| \le 2c-4$, the number of maximal induced forest that contain $e$ and $f$ is at most $\f(2c-4)$.  	  

	\medskip
	\eqref{P4}
	We proceed in a similar fashion to prove the second part of the lemma.
	Let $a,b,c$ be three vertices that induce a $P_3$ in $G$.
	We count the number of maximal induced forests containing $a,b,c$ and in which $c$ is not a leaf.
	Since $c$ is not a leaf, any such forest contains a vertex $d$ that is adjacent to $c$ but not to $a,b$.
	It is not hard to see that there are at most $c-1$ possible choices for $d$ since $d\in V(G)\setminus N[a,b]$,
	and $|V(G)\setminus N[a,b]| < c$.
	Let $S$ be a set inducing a maximal forest and containing $a,b,c,d$ in $G$.
	To prove the lemma we show that any such set $S$ also induces a maximal forest in a graph on $2c-3$ vertices.

	Consider an arbitrary vertex $v \in S$. 
	Since $G[S]$ is a tree containing a $P_4$ induced by $\{a,b,c,d\}$ it follows that $v$ is either non-adjacent to both $a$ and $b$, or non-adjacent to $c$ and $d$.
	Therefore, $S$ is also a maximal induced forest in the graph $H = G[ (V(G)\setminus N[a,b]) \cup (V(G)\setminus N[c,d]) \cup \{b,c\}]$.
	As $G$ is the complement of a $c$-closed graph there are at most $c-1$ vertices non-adjacent to both $a,b$ (including $d$)
	and at most $c-1$ vertices non-adjacent to $c,d$ (including $a$), i.e., $|V(H)| \le 2c$.
	Similarly as before, we are only interested in sets $S$ containing $a,b,c,d$.
	Let $H'$ be the graph obtained by contracting the edges $ab, bc, cd$ in $H$ into a vertex $u$.
	It is not hard to see that $S\cup u \setminus \{a,b,c,d\}$ induces a maximal induced forest in a graph $H'$, and the same holds for any set inducing a maximal forest that contains $\{a,b,c,d\}$.
	Since $|V(H')| \le 2c-3$, it follows that the number of maximal induced forests that contain $a,b,c$ is at most $(c-1)\cdot \f(2c-3)$.
\end{proof}
The main idea in the both parts of the above proof is finding two edges $e$ and $f$ 
that partition the rest of the graph into their respective non-neighborhoods.
This idea is generalized in Lemma~\ref{lemma:kEdges} and will be used in later proofs.

\begin{theorem}[Forests]
\label{thm:forests}
	Let $G$ be the complement of a $c$-closed graph. 
	The number of maximal induced forests in $G$ is at most 
	$$ n^3\m_0(c-1) + 2n^3 \cdot (c-1) \f(2c-3) \le 3 n^3 \cdot (c-1) \cdot 1.8638^{2c-3}\,.$$
\end{theorem}
\begin{proof}
	A forest either contains at least two non-trivial components, a $P_4$, or is a star.
	By Lemma~\ref{lemma:maximumNumOfStars} there are at most $n^3 \m_0(c-1)$ maximal stars in $G$. 
	By Lemma~\ref{lemma:forestsP4} the number of maximal induced forests that 
	contain at least two non-trivial components or a $P_4$ is at most $2n^3 \cdot (c-1) \f(2c-3)$.
	The theorem follows since $3^{\frac{c-1}{3}} \le 1.443^{c-1} \le 1.8638^{2c-3}$ for any integer $c$ bigger than $1$.
\end{proof}

\smallskip \noindent \textbf{Enumeration}
By Theorem~\ref{thm:forests},
the polynomial delay algorithm for enumerating maximal induced forest~\cite{conte2019new,conte2019proximity} 
on the complement of $c$-closed graph runs in FPT time. 
As the enumeration algorithm takes $O(n^5)$ per maximal forest we obtain an FPT algorithm.
Similarly, as in Corollary~\ref{cor:fasterBoundedDegree}, 
we can obtain a better running time by applying the algorithm directly in the proof of the upper bound. 
We state the improved running time in the following corollary.

\begin{corollary}
	For $c$-closed graphs, there is an FPT algorithm running in time $O(n^3 \cdot 1.8638^{2c-3} \cdot c^6)$ for {\sc Enumerate co-forests.}
\end{corollary}

\bibliography{references}

\end{document}